\documentclass{amsart} 

\usepackage{amssymb,amsmath,amsfonts,latexsym} 
\usepackage{mathrsfs}
\usepackage{hyperref}
   
\usepackage{mathtools}
\usepackage{multirow,array}
\usepackage{graphicx}
\usepackage{subfigure}
\usepackage{epstopdf}
\usepackage{float} 
\usepackage{color, xcolor}
\usepackage{enumitem}

\allowdisplaybreaks

\def\cpt#1{{_{#1}^C\!D}}
\def\R{{\mathbb{R}}}
\def\N{{\mathbb{N}}}
\def\cF{{\mathcal F}}

\def\scSp{{\mathscr S}^\prime}
\def\C{{\mathbb{C}}}
\def\scS{{\mathscr S}}
\def\scSp{{\mathscr S}^\prime}
\def\cL{{\mathcal L}}
\def\scD{{\mathscr D}}
\def\cpt#1{{_{#1}^C\!D}}
\def\cA{{\mathcal A}}
\def\Z{{\mathbb{Z}}}
\def\cB{{\mathcal B}}
\def\cC{{\mathcal C}}
\def\WFF{\operatorname{WF}}
\def\jap#1{\langle {#1} \rangle}
\newcommand{\Char}{\operatorname{Char}}
\newcommand{\ta}{\tau}
\newcommand{\x}{\langle x\rangle}
\newcommand{\csi}{\langle \xi \rangle}
\newcommand{\Op}{\operatorname{Op}}
\newcommand{\supp}{\operatorname{supp}}
\newcommand{\norm}[1]{\langle#1\rangle}
\newcommand{\<}{\langle}
\renewcommand{\>}{\rangle}
\newcommand{\loc}{\operatorname{loc}}
\DeclareMathOperator*{\vspan}{span}

\newtheorem{theorem}{Theorem}[section]
\newtheorem{proposition}[theorem]{Proposition}
\newtheorem{lemma}[theorem]{Lemma}
\newtheorem{cor}[theorem]{Corollary}

\theoremstyle{definition}
\newtheorem{example}[theorem]{Example}

\theoremstyle{remark}
\newtheorem{remark}[theorem]{Remark}
\renewcommand{\theta}{\vartheta}

\begin{document}


\title[Regularity and decay of solutions for sub-diffusion type equations]{Representation formula, regularity and decay\\of solutions for sub-diffusion type equations}


\author[S. Coriasco, G. Girardi, S. Pilipovi\'c]{
        Sandro Coriasco 
\and
        Giovanni Girardi 
\and 
Stevan Pilipovi\'c
 }

\address{
Dipartimento di Matematica ``G. Peano'', Universit\`a degli studi di Torino, Torino, Italy
}
\email{sandro.coriasco@unito.it} 
\address{
Dipartimento di Ingegneria e Scienze, Università Telematica Universitas Mercatorum, Piazza Mattei, 10 - 00186 Roma, Italy \\
}
\email{giovanni.girardi@unimercatorum.it}
\address{
Department of Mathematics and Informatics, Faculty of Sciences,
		University of Novi Sad, Trg D. Obradovi\'ca 4, 
		RS-21000 Novi Sad, Serbia\\
}
\email{pilipovic@dmi.uns.ac.rs}



\begin{abstract}
{We study regularity and decay properties for the solutions of the Cauchy problem for time-fractional partial differential equations,
	with tempered initial data, belonging to suitable (weighted) Sobolev spaces,
	associated with a differential operator on space variables with polynomially bounded coefficients. 
	We obtain a representation formula for the solution, modulo time-regular functions, smooth and rapidly decreasing with respect to
	the space variables. By means of the representation formula, the (decay and smoothness) singularities of the solution of the homogeneous 
	Cauchy problem can be controlled, in terms of (global) wavefront sets of the initial data. }
\end{abstract} 
\maketitle


\section{Introduction}\label{sec:intro}
We consider the Cauchy problem for a non-homogeneous subdiffusive heat equation, namely
\begin{equation}
\label{eq:CPmain}
\begin{cases}
\partial_t^r u(t,x) + \Op(a)u(t,x)= f(t,x), \quad (t,x)\in (0,+\infty)\times \R^d,&\\
u(0,x)=u_0, \quad x\in \R^d. &
\end{cases}
\end{equation}
In \eqref{eq:CPmain}, $r$ is a positive real number in $(0,1)$ and $\partial_t^r u$ denotes the (forward) Caputo fractional derivative of order $r\in (0,1)$ with respect to the time variable $t$, with starting time $0$ (cf., for instance, \cite{Kilbas}),  defined by 
\begin{equation*}
\partial_t^r u(t,x)= \cpt0_t^{r}u(t,x)=\frac{1}{\Gamma(1-r)}\int_0^t \frac{\partial_t u(\tau,x) }{(t-\tau)^r}\,d\tau.
\end{equation*} 
Recall that, more generally, for $\nu\in(0,+\infty)\setminus\N$, $c\in\R$, 
the (forward) Caputo fractional derivative of order $\nu$ with respect to $t$, with starting time $c\in\R$, is defined by
\begin{equation}\label{eq:fracdergen}
	\cpt c_t^{\nu}g(t)=\frac{1}{\Gamma(1+[\nu]-\nu)}\int_c^t \frac{g^{([\nu]+1)}(\tau)}{(t-\tau)^{\nu-[\nu]}}\,d\tau,\quad t\in(c,+\infty).
\end{equation}
$\Op(a)$ in \eqref{eq:CPmain} is a hypo-elliptic (pseudo)differential operator with symbol $a=a(\xi)$ or $a=a(x,\xi)$. 
We postpone its precise definition and the hypotheses on $a$, for the two main symbol and operator classes we consider, and first discuss
some background and motivating facts.

The interest to study the model \eqref{eq:CPmain} comes from the pioneering work \cite{Nigmatullin1984}. 
Here, the author introduces a diffusion equation with memory which allows to take into account the non-Markovian 
character of the excitation transfer process in some heterogeneous media. It takes the general form
\begin{equation}\label{eq:CPconstant-coefficients}
	\frac{\partial^\beta U(t,x)}{\partial t^\beta} = C = \frac{\partial^2 U(t,x)}{\partial x^2}, \quad \beta\in (0,2],\quad  (t,x)\in [0,T)\times \R^d,
\end{equation}

where $C$ is a positive constant related to the diffusion coefficient anisotropy. The case $\beta=1$ describes the usual diffusion (absence of memory), which occurs, 
for instance, in a strongly dispersive medium, whereas in the case $\beta=2$ we find the classical wave equation (which corresponds to a full memory), which describes 
the transfer process in an homogeneous medium in which no energy loss appears. 
A detailed review about the properties of the solution to the Cauchy-type problem associated to \eqref{eq:CPconstant-coefficients} (in space dimension $d=1$) 
can be found in Chapter 6 of \cite{MainardiBook}. We also mention that long time decay estimates for the Cauchy problem associated to \eqref{eq:CPspace-time-fractional} were studied in \cite{DA2019} for $\beta\in (0,1)$ and in \cite{DAEP2019} for $\beta\in (1,2)$, and in both such cases they were applied to study the influence of a non-linear perturbation. Similar issues were later discussed in \cite{DAG2022}, in presence of an additional term $\partial_t U/\partial t$ in equation \eqref{eq:CPconstant-coefficients}, which can be interpreted as an heat equation with fractional damping (see also \cite{DAG2023} for other multi-term time-fractional models). 

In \cite{GMSR2021} the authors derive the following space-time fractional diffusion equation
\begin{equation} 
\label{eq:CPspace-time-fractional}
\frac{\partial^\beta U(t,x)}{\partial t^\beta}=\frac{\partial^\alpha U(t,x)}{\partial x^\alpha}, \quad \beta\in (0,2],\quad (t,x)\in [0,T)\times \R,
\end{equation}
where the time-fractional derivative is defined in the Caputo sense, whereas the space-fractional derivative of order $\alpha\in (0,2]$ is defined as a pseudodifferential operator with symbol $a(\xi)=-|\xi|^\alpha$, $\xi\in \R$. In particular, they show that such equation governs a large class of stochastic processes which are useful for modeling the dynamics of financial markets and for risk management. 
The Green functions for problem \eqref{eq:CPspace-time-fractional} can be expressed in terms of Wright type functions (see \cite{GIL2000}) and 
interpreted as probability density functions (see \cite{Schneider89}).
A review about further studies regarding problem \eqref{eq:CPspace-time-fractional} can be found in Chapter 7 of \cite{MainardiBook}.

In \cite{GLU2000} the authors consider a generalization of the models \eqref{eq:CPconstant-coefficients} and \eqref{eq:CPspace-time-fractional}, given by the pseudodifferential equation of fractional order
\begin{equation}\label{eq:CP-pseudiff}
	\frac{\partial^\beta U(x,t)}{\partial t^\beta} = a(D)U(t,x), \quad (t,x)\in [0,\infty)\times \R^d,
\end{equation}
where $a(D)=\Op(a)$ is a pseudodifferential operator, possibly with  a singular symbol. Namely, $a(\xi)$ is a continuous function in an open domain $G\subset\R^d$. The authors in \cite{GLU2000} obtain a representation formula of the solution to problem \eqref{eq:CP-pseudiff} in terms of Mittag-Leffler functions $E_{\beta,1}=E_{\beta,1}(z)$ and they apply it to study well-posedness results in the space $\Psi_{G,p}(\R^d):=\{f\in L^p(\R^d): \supp\widehat{f}\subset G\}$, $1\leq p\leq \infty$, endowed with a suitable notion of convergence, and in its dual space. 
The application of these results allows also to obtain some information about the well-posedness of problem \eqref{eq:CP-pseudiff} in the classical Sobolev spaces $H^s(\R^d)$.

The motivation for the analysis of equations of the form \eqref{eq:CP-pseudiff} comes from a general energy balance law  with the appropriate constitutive equations,
depending on a material or a substance or a field.
For example, stress and strains in visco-elastic bodies
or various fields in Maxwell's equations. We also mention that  the fractional Zener's and Burger's type model, related to
stress and strains in visco-elastic bodies various constitutive equations, were analyzed in \cite{AMP,ANPR}. 

In this paper, we find a representation formula for the solution to \eqref{eq:CPmain}, in terms of derivatives of the Mittag-Leffler functions, under suitable 
assumptions on the symbol $a$, by means of a (parameter-dependent) parametrix construction, in the case the symbol $a$ depends on $x$, or by means of 
(parameter-dependent) inversion, when the symbol $a$ does not depend on $x$. Laplace transform of vector-valued distributions, as well as its interplay
with pseudodifferential operators, is employed here (see Section 4 of \cite{CGP2} for details).
We then apply such representation to obtain information on the 
regularity and decay properties of the solution, for initial data belonging to appropriate (weighted) Sobolev spaces. The obtained results rely, in particular, on certain 
decay properties of the derivatives of the Mittag-Leffler functions, which are established in Section \ref{subs:M-L}. 
As recently discussed in 
\cite{Garrappa-Mainardi2025}, the study of these properties is also crucial for analyzing the behaviour of solutions to multi-term fractional-order differential equations, 
which can indeed be expressed in terms of derivatives of the Mittag-Leffler functions.

The paper is organized as follows. Our main results are presented in Sections \ref{sec:constcoeff} and \ref{sec:varcoeff}.
In Section \ref{sec:constcoeff} we focus on the constant coefficients (that is, \textit{Fourier multipliers}) case $\Op(a)=a(D)$.
In Section \ref{sec:varcoeff} we switch to the much more challenging variable coefficients case $\Op(a)=a(\cdot,D)$, and prove
a representation formula for the solution for two relevant classes of symbols. Section \ref{subs:M-L} collects various properties of Mittag-Leffler
functions, including estimates of the derivatives tailored to the symbolic analysis developed in this paper.
To keep this exposition within a reasonable length, the construction of the parameter-dependent parametrix and some properties
and results about the Laplace transform are described in detail in \cite{CGP2}, where basic notions concerning the so-called SG-calculus
can be found as well.
We employ the standard notation $D=(D_1,...,D_d),$ where $D_j=-i\partial_{x_j}$, $i=\sqrt{-1}$, $j=1,...,d$, for the derivatives,
and $\widehat u=\cF u = \cF(u)$ for the Fourier transform, of functions and distributions.

\section*{Acknowledgements}
 The first author has been partially supported by the Italian Ministry of the University and Research - MUR, within the framework of the Call relating to the scrolling of the final rankings of the PRIN 2022 - Project Code 2022HCLAZ8, CUP D53C24003370006 (PI A. Palmieri, Local unit Sc. Resp. S. Coriasco). The first author also expresses
gratitude for the hospitality extended to him during his visit to the Department of Mathematics and Informatics, University of Novi Sad, Serbia, during A.Y. 2024/2025,
where part of this work was developed. The second author has been partially supported by INdAM GNAMPA Project, Grant Code CUP E55F22000270001.
The third author has been supported by the Serbian Academy of Sciences and Arts, project F10.

\section{Constant coefficients equations}\label{sec:constcoeff} 
In this section we give a simple construction of a solution of the Cauchy-type problem \eqref{eq:CPmain},
where the symbol $a$ does not depend on $x$. Recall that, in this case, $\Op(a)=a(D)$ is called Fourier multiplier.
If $a$, in particular, is a polynomial, then $a(D)$ is a partial differential operator with constant coefficients. 
\begin{example}
First, we give an example inspired by \cite{AMP}. The energy balance law for the heat conduction reads
$$
\partial_te(t,x)=-K \mbox{ div}_x \mathbf{q}(t,x),\quad t\in(0,\infty),\; x\in\mathbb R^d,
$$
where $K>0$ is the coefficient of diffusion, $e$ is the internal energy 
and $\mathbf q$ is the heat flux vector.  Let $T$ be the temperature and assume that $e(t,x)=cT(t,x)+T_0(x), x\in\mathbb R^d, t\geq 0,$
where the constant $c$ is the specific heat and $T_0(x)=T(0,x).$ Using this, we arrive to 
\begin{equation}\label{eq:heatflux}
\partial_tT(t,x)=-\frac{K}{c}\mbox{ div}_x \mathbf{q}(t,x),\quad t\in(0,\infty),\; x\in\mathbb R^d.
\end{equation}
Next, instead of the classical Fourier law for the constitutive equation, namely
\[
	\mathbf q(t,x)=-c_0\nabla_x T(t,x), 
\]
where $c_0$ is the heat diffusion constant,
many authors use different forms of constitutive equations, see for example \cite{AJP,APS}. 
Let $r=1-\beta$, $\beta\in(0,1)$. We propose a constitutive equation of the form
\begin{equation}\label{eq:constlaw}
_0D_t^{-\beta}\mathbf q(t,x)=\nabla_x T(t,x),\quad   t\in(0,\infty),\; x\in\mathbb R^d,
\end{equation}

\noindent where $_0D_t^{-\beta}\mathbf{q}(t,x)$ is obtained by means of the Riemann-Liouville integral, namely
 \[
 	_0D_t^{-\beta}f(t):= \frac{1}{\Gamma(\beta)}\int_0^t \frac{f(\tau)}{(t-\tau)^{1-\beta}}d\tau, t>0.
\]
By \eqref{eq:constlaw}, applying $_0D_t^{-\beta}$ to both sides of \eqref{eq:heatflux} and setting $k=K/c$, it follows
\begin{equation}\label{eq:toymodel}
\partial_t^r T(t,x)= _0D_t^{-\beta}\partial_tT(t,x)=-k\Delta_x T(t,x),\quad  t\in(0,\infty),\; x\in\mathbb R^d.
\end{equation}
Comparing with (\ref{eq:CPmain}), in the given example we find
$ \Op(a)=a(D)$, $a(\xi)=-|\xi|^2$, $f\equiv0$, $T(0,x)=T_0(x)$.
\end{example}
Let then $r\in(0,1)$ and let $a(x,\xi)=a(\xi)$, $\xi\in\R^d$, be a nonnegative continuous function.
With these choices, \eqref{eq:CPmain} assumes the form
\begin{equation}\label{eq:Furmultipl}
\partial ^r_tu(t,x)+a(D)u(t,x)=f(t,x),\; u(0,x)=u_0(x), \quad t\in(0,\infty),\; x\in\mathbb R^d.
\end{equation}
For the main result of this section we first recall that the family 
$f_{\alpha}(t)=t_+^{\alpha-1}/\Gamma(\alpha)$, $\alpha>0$, $t\in\R$, and $f_{\alpha}=f^{(N)}_{\alpha+N}, \alpha\leq 0$, where $t^\alpha_+=H(t)t^\alpha$, 
is a group, that is $f_{\alpha}*f_{\beta}=f_{\alpha+\beta}$, $\alpha, \beta\in \R$ (see \cite{V}).
Moreover, for the Laplace transform, there holds
\begin{equation*}
 \mathcal{L}\left(\frac{t_+^{\alpha-1}}{\Gamma(\alpha)}\right)(s)=%
\frac{1}{s^{\alpha}}, \ \alpha\in\R,\  \Re
s>0. 
\end{equation*}
%
Recall also \(\mathcal{L}(\delta)(s)=1\),  \(s\in\C\).
We will also need the space 
\[
	\scS([0,\infty)\times\R^d)=\bigcap_{k\in \N_0} \scS_k([0,\infty)\times\R^d),
\]
where
\begin{align*}
\scS_k([0,\infty)\times\R^d) = & \{\phi\in C^\infty([0,\infty)\times\R^d)\colon \\
&
\gamma_k(\phi)=\sup_{(t,x)\in[0,\infty)\times\R^d,p+|q|+\alpha+|\beta|\leq k}|t^p x^q \partial_t^\alpha\partial_x^\beta \phi(t,x)|<\infty\}.
\end{align*}
The space $\scS([0,\infty)\times\R^d)$ is a closed subspace of $\scS(\R^{d+1})$. As such, it is an FS and Montel space 
(see \cite{T}). We refer to \cite{JP} for the definition of $\scS'([0,\infty)\times \R^d)$. Note that 
$f_{\alpha+1}\in\scSp_k([0,\infty))$ for $\alpha<k$, $\alpha\in\N_0, k\in \N$, and  $f^{(\alpha+1)}_{\alpha+1}=\delta$.

\medskip

Let us also recall a simple result connected with the definition \eqref{eq:fracdergen} of  $\cpt c_t^{\nu}g$.
\begin{lemma}
Let $\nu=[\nu]+r>0$ and $g\in C^{[\nu]}(0,\infty)$ so that $g^{[\nu]}$ is locally absolutely
continuous in $[0,\infty)$.
Then, $\cpt c_t^{\nu}g(t)=\cpt c_t^{[\nu]+r}g(t)\in L^1_{\loc}[0,\infty)$.
\end{lemma}
\begin{proof}
It is enough to observe that, for any $A>0$,
\[
	\int_0^A\int_0^t\frac{g^{([\nu]+1)}(\tau)}{(t-\tau)^r}d\tau dt=
	\int_0^Ag^{([\nu]+1)}(\tau)\left[\int_0^A\frac{H(t-\tau)}{(t-\tau)^r}dt\right]d\tau <\infty.
\]
\end{proof}
\begin{lemma}\label{lem:Spstruct} 
Let $v\in\scSp([0,\infty)\times \R^d)$. Then there exist continuous functions $V_1(t,x)$ and $V_2(t,x)$, of polynomial growth in both variables $t\in[0,\infty)$, 
$x\in\R^{d}$, and $\alpha,m\in\N_0$, such that
$$
v= D_t^{\alpha+1} \Delta^m_x V_1 + D^{\alpha+1}_tV_2, \; \;\supp V_{1},  \supp V_2 \subset [0,\infty)\times \R^d.
$$
\end{lemma}
\begin{proof}
Since $v\in\scSp([0,\infty)\times\R^d)$, it follows that $v\in\scSp_k([0,\infty)\times\R^d)$ for some $k\in\N_0$. Recalling Schwartz's parametrix method, there exist  
$m, \alpha\in \N$,  $h\in C_0^{m}(\R^d)$ and $r\in\scD(\R^d)$ such that 
$ \delta=\Delta^m h+r$. It follows
\begin{align*}
	v&=v*\delta=[v*(\partial_t^{\alpha+1}f_{\alpha+1}\otimes(\Delta^m_x h+r))](t,x)
	\\
	&=\partial_t^{\alpha+1}\Delta^m_x\langle v(\tau,y),f_{\alpha+1}(t-\tau)\cdot h(x-y)\rangle \\
& \hspace{50pt}+\partial_t^{\alpha+1}\langle v(\tau,y),f_{\alpha+1}(t-\tau)\cdot r(x-y)\rangle.
\end{align*}
Since $V_1(t,x)=\langle v(\tau,y),f_{\alpha+1}(t-\tau)\cdot h(x-y)\rangle$
and $V_2(t,x)=\langle v(\tau,y),f_{s+1}(t-\tau)\cdot r(x-y)\rangle$ are continuous functions with the desired properties, the proof is completed.
\end{proof}
\begin{remark}
Lemma \ref{lem:Spstruct} allows us to choose $f\in\scSp([0,\infty)\times\R^d)$ in \eqref{eq:Furmultipl}. Indeed, the Laplace transform of $f$ with respect
to $t$ is then analytic in $\Re s>0$. The same property holds true for its partial Fourier transform with respect to $x$.
\end{remark}
Since $e_s\colon t\mapsto e^{-st}$, $t\geq 0$, $\Re s>0$, belongs to $\scS([0,\infty))$, it is possible to define,
for $f\in\scSp([0,\infty)\times\R^d)$,
\[
	(\cL f)(s,x)=\langle f(t,x),e^{-st}\rangle=\langle f(t,x),\zeta(t)e^{-st}\rangle, \quad\Re s>0, x\in\R^d,
\]
where $\zeta\in C^\infty(\R)$ is supported in $[-\varepsilon,\infty)$ and equals one in $[-\varepsilon/2,\infty)$, $\varepsilon>0$.
The definition of $\cL f$ does not depend on $\varepsilon>0$ and $\zeta$ with the desired properties 
(see also Section 4 of \cite{CGP2}, Definition 4.6).

Two distributions $h,k\in\scSp(\R^d)$ are called convolvable if their convolution, defined by 
$$
\langle h*k,\phi\rangle=\lim_{\nu\rightarrow \infty}
\langle (h\otimes k)(x,y),\kappa_\nu(x,y)\phi(x+y) \rangle, \quad\phi\in\scS(\R^d),
$$ 
exists independently of a unit sequence $(\kappa_\nu)_{\nu\in\N}\subset\scD(\R^{2d})$, whose elements $\kappa_\nu$
equal one in balls $B(0,R_\nu)$ and zero out of balls $B(0,R_{\nu+1})$, $\nu\in\N$, 
where $R_\nu\rightarrow \infty$, $(R_\nu)_\nu$ is strictly increasing (see \cite{V}).

Let 
 $a(\xi), \xi\in\R^d,$ be a  non-negative continuous function of slow growth over $\R^d,$   (that is, polynomially bounded)
and
$E(t,\xi)=E_{r,1}(-(a^{1/r}(\xi)t)^r)$, $x\in\R^d, t\in[0,\infty)$, where $E_{r,1}$ denotes the Mittag-Leffler function (see Section \ref{subs:M-L}). It is a continuous function on the domain $x\in\R^d, t\in[0,\infty)$ We define $E(x,t)=0$ for $x\in\R, t\in(-\infty,0)$. Below, we  will use notation $F(t,x)=\cF^{-1}_{\xi\to x}(E(t,\xi))$. The next Theorem \ref{thm:maint1} is our first main result.

\begin{theorem}\label{thm:maint1}
Let  $f\in\scSp([0,\infty)\times\R^{d})$ and $a\in C(\R^d)$ be a non-negative continuous function of slow growth. 
Let $u_0\in\scSp(\R^d)$ be convolvable with $\mathcal  F^{-1}_{\xi\rightarrow x}(\langle E(t,\xi),\theta(t)\rangle)$ for every $\theta\in\scS([0,\infty))$.
Assume also that $(F(\cdot,x)* f_r(\cdot))(t,x)=F(t,x)*_t f_r(t)$ is $(t,x)$-convolvable with $\frac{\partial f}{\partial t}$.
Then the unique solution of \eqref{eq:Furmultipl} in $\scSp([0,\infty)\times\R^d)$ is given by
\begin{equation}\label{eq:5}
u(t,x)=F(t,x)*_xu_{0}(x)+
(F(t,x)*_tf_r(t))*_{(t,x)}\frac{\partial f}{\partial t}(t,x).
\end{equation}

Moreover, assume that for $m\in\N$ and $l\in\{1,...,m\}$,
\begin{equation}\label{eq:cond0}
\begin{aligned}
& u_0\in C^m(\R^d), \\
& \frac{\partial f}{\partial t}(\cdot,x)\in L^1_{\loc}([0,\infty)), \; \text{ for all }
x\in \R^d, \\
& \frac{\partial f}{\partial t}(t,\cdot)\in C^m(\R^d), \quad \text{ for almost all } t\in[0,\infty),
\end{aligned}
\end{equation}
 and that $a$ is of slow growth. 
Let $k\in\N$ satisfy $(k+1)r>1$.  
 If 
\begin{equation}\label{eq:cond1}
\cF^{-1}_{\xi\rightarrow \cdot}[\xi_j^la^p(\xi)\widehat u_0(\xi)]\in C(\R^d), 
\end{equation} 
for all $p=1,...,k+1, j=1,...,d, l=1,...,m,$ and 
\begin{align}\label{eq:cond2}
\cF^{-1}_{\xi\rightarrow \cdot}\left[\xi_j^la^p(\xi)\cF_{x\rightarrow \xi}\left(\frac{\partial f}{\partial t}(t,x)\right)\right]\in C(\R^d),
\end{align}
for all  $t\in[0,\infty),\; p=1,...,k+1, \; j=1,...,d, l=1,...,m,$ then the solution  $u$, given by \eqref{eq:5},
has locally integrable derivative with respect to $t\geq 0$ for every $x\in\R^d$. Moreover, it is of class $C^m$ with respect to $x\in\R^d$ for every $t\geq 0,$ and $\frac{\partial u}{\partial t}$ is of class $C^m$ for almost all $t\geq 0.$
\end{theorem}
\begin{remark} The explicit form \eqref{eq:5} of the solution enables us to presume various conditions on $a$, $u_0$ and $f$ 
so that conditions \eqref{eq:cond1} and \eqref{eq:cond2} hold and \eqref{eq:Furmultipl} has a unique  solution with the locally integrable derivative with respect to $t$ and with increased regularity with respect to $x$. If $u_0=0$ then, with appropriate assumptions on $f,$ one can have a solution with prescribed regularity properties with respect to  both variables.
\end{remark}
\begin{proof}[of Theorem \ref{thm:maint1}] 
Note that $(\cL f)(s,\cdot)$, $\Re s>0$,  is analytic in $s$, $\Re s>0$ with values in $\scSp(\R^d)$. 
So, applying the Fourier transform one obtains $\widehat{\cL f}(s,\cdot)$, analytic in $s$, $\Re s>0$, and belonging to $\scSp(\R^{d})$ for every $s$, $\Re s>0$. 
Applying Fourier and then Laplace transform to \eqref{eq:Furmultipl}, one obtains
\begin{equation*}
\partial_t^r\widehat u(t,\xi)+a(\xi)\widehat u(t,\xi)=\widehat f(t,\xi),
\end{equation*}
and
%
\[
	(s^r+a(\xi))U(s,\xi)=s^{r-1}\widehat u_0(\xi)+(\cL{\widehat{f}})(s,\xi),\quad  \Re s>0, \xi\in\R^d,
\]
that is,
\[
	U(s,\xi)=\frac{s^{r-1}}{s^r+a(\xi)}\widehat u_0(\xi)+\frac{(\cL\widehat{f})(s,\xi)}{s^r+a(\xi)}, \quad \Re s>0, \xi\in\mathbb R^d.
\]

Then, by inverse Laplace transform (see, e.g., Section 4 of \cite{CGP2}),
we obtain
\[
	U(t,\xi)=\cL^{-1}_{s\to t}\left(\frac{s^{r-1}}{s^r+a(\xi)}\right)\widehat u_0(\xi)+\cL^{-1}_{s\to t}\left(\frac{1}{s^r+a(\xi)}\right)*_t\widehat{f}(t,\xi),
\]
for all $t\in(0,\infty), \xi\in\mathbb R^d.$
Employing the notation of Mainardi \cite{M}, we have
\begin{equation}\label{ML}
	\cL^{-1}_{s\to t}\left(\frac{s^{r-1}}{s^r+a(\xi)}\right)(t)=e_r(a^{1/r}(\xi)t)=E_{r,1}(-(a^{1/r}(\xi)t)^r)=E(t,\xi),
\end{equation}
for all $t\in(0,\infty), \xi\in\mathbb R^d.$
This is a consequence of the fact that
$$-\frac{s^{r-1}}{s^r+a(\xi)}=\frac{a(\xi)}{s(s^r+a(\xi))}-\frac{1}{s}, \Re s>0
$$
and that $\mathcal L^{-1}(\frac{1}{s})(t), t\in\R,$ is the Heaviside function (for which we assume the continuity from the right at $t=0$).
We also have for any  $\xi\in\R^d$
\[
	\cL^{-1}_{s\to t}\left(\frac{1}{s^r+a(\xi)}\right)(t)=\mathcal L^{-1}_{s\to t}\left(\frac{s^{r-1}}{s^r+a(\xi)}\cdot\frac{1}{s^{r-1}}\right)(t)=
	E(t,\xi)*_t\left[\frac{d}{dt}\frac{t^{r-1}}{\Gamma(r)}\right](t), 
\]
where on the right hand side we have the convolution of two distributions supported in $[0,\infty).$
By (3.3) in \cite{M}, there holds,
\begin{equation}\label{eq:4}
	E(t,\xi)\sim\frac{a^{-1}(\xi)t^{-r}}{\Gamma(1-r)},
	\;\text{ and }\;
	E(t,\xi)\sim 1- \frac{a(\xi)t^{r}}{\Gamma(1+r)}, \mbox{ as }  a(\xi)^{1/r}t\rightarrow 0,
\end{equation}
as $a^{1/r}(\xi) t\rightarrow \infty$. By \eqref{eq:4}, for any given $\varepsilon>0$ there exists $M>0$ such that
\[
	E(t,\xi)\cdot\left[ \frac{a^{-1}(\xi)t^{-r}}{\Gamma(1-r)}\right]^{-1}\in(1-\varepsilon, 1+\varepsilon),\quad a^{1/r}(\xi)t>M,
\]
which implies
\[
	E(t,\xi)\leq \frac{a^{-1}(\xi)t^{-r}}{\Gamma(1-r)}(1+\varepsilon)\leq \frac{1+\varepsilon}{\Gamma(1-r)\,M^r},\quad a^{1/r}(\xi)t>M.
\]
Similarly, for any given $\eta>0$ there exists $\delta>0$ such that
\[
	E(t,\xi)\cdot\left[ 1- \frac{a(\xi)t^{r}}{\Gamma(1+r)}\right]^{-1}\in(1-\eta,1+\eta),\quad a(\xi)^{1/r}t<\delta.
\]
This gives
\[
	E(t,\xi)\leq \left[1+ \frac{\delta^r}{\Gamma(1+r)}\right](1+\eta),\quad a(\xi)^{1/r}t<\delta.
\]
Denote
\begin{align*}
	&\cA_1=\{(t,\xi): t\geq 0,\, a(\xi)=0\},  &&\cA_2=\{(t,\xi): t\geq 0,\, a^{1/r}(\xi)t>M)\},
	\\
	&\cA_3=\{(t,\xi): t\geq 0,\, a^{1/r}(\xi)t<\delta)\},   &&\cA_4=\{(t,\xi): t\geq 0,\, a^{1/r}(\xi)t\in[\delta,M]\}.
\end{align*}
Since $E(t,\xi)$ is bounded on all measurable sets $\mathcal A_1, \mathcal A_2, \mathcal A_3, \mathcal A_4,$
it follows  that $E(t,\xi)\in L^\infty([0,\infty)\times\mathbb R^d).$

With this and the assumptions of the first part of the statement, we have that the solution of \eqref{eq:Furmultipl} belongs to $\scSp([0,\infty)\times\R^d)$,
as claimed. The uniqueness follows from the injectivity of the Fourier and Laplace transforms on $\scSp(\R^d)$ and $\scSp([0,\infty))$, respectively. To prove the second part of the claims we need the subsequent Lemma \ref{lem:l2}.
\begin{lemma}\label{lem:l2} Under the assumptions \eqref{eq:cond0}, \eqref{eq:cond1}, and \eqref{eq:cond2}, with $a$ as in Theorem \ref{thm:maint1}  and $k\in\N$ satisfying $(k+1)r>1$, it follows that
$\dfrac{\partial F}{\partial t}(t,x )*_xu_0(x)$, $x\in\R^d$, $t\in[0,\infty),$ is a locally integrable function with respect to $t$ for all $x\in\R^d$ and of class $C^m(\R^d), m\in\N,$ with respect to $x$ for almost all $t\in(0,\infty)$.
\end{lemma}
\noindent
Postponing the proof of Lemma \ref{lem:l2}, let distributions $u_1$ and $u_2$ be defined as  $u_1(t,x)=F(t,x)*_xu_0(x)$ and $u_2(x,t)=u(x,t)-u_1(x,t)$. Let $j\in\{1,...,d\}$ and $l\in\{1,...,m\}$.
One has 
\begin{align*}
	D_j^l\frac{\partial u_1}{\partial t}(t,x)
	&=\frac{\partial F}{\partial t}(t,x)*_xD_j^l u_0(x)
	\\
	&=\cF^{-1}_{\xi\rightarrow x}\left[\cL^{-1}_{s\to t}\left(\frac{s^r}{s^r+a(\xi)}\right)\xi_j^l\,\widehat{u_0}(\xi)\right]\\
&
	=\cF^{-1}_{\xi\rightarrow x}\left[\cL^{-1}_{s\to t}\left(\frac{1}{1+s^{-r}a(\xi)}\right)\xi_j^l\,\widehat{u_0}(\xi)\right]
\\
	&=\sum_{p=0}^k \mathcal F^{-1}_{\xi\rightarrow x}\left[\cL^{-1}_{s\to t}\left((-1)^p s^{-rp}a^p(\xi)\right)\xi^l_j\widehat{u_0}(\xi)\right]
	\\
	&+\cF^{-1}_{\xi\rightarrow x}\left[\cL^{-1}_{s\to t}\left(\frac{(-1)^{k+1}s^{-(k+1)r}a^{k+1}(\xi)}{1+s^{-r}a(\xi)}\right)\xi_j^l\widehat{u_0}(\xi)\right]
\\
	&=\sum_{p=1}^k (-1)^p f_{rp}(t)\mathcal F^{-1}_{\xi\rightarrow x}(\xi^l_ja^{p}(\xi)\widehat{u_0}(\xi))\\
& \qquad +
	(-1)^{k+1}\cF^{-1}_{\xi\rightarrow x}\left[\cL^{-1}_{s\to t}\left(\frac{1}{s^{kr}(s^r+a(\xi))}\right)\xi_j^la^{k+1}(\xi)\widehat{u_0}(\xi)\right].
\end{align*}
Since $\cL^{-1}_{s\to t}\left(\frac{1}{s^{kr}(s^r+a(\xi))}\right)(t,\xi)$ is continuous in $t\geq 0$  and bounded in $\xi\in\R^d$, by Lemma \ref{lem:l2} and \eqref{eq:cond1} it follows that $u_1$ has a locally integrable derivative with respect to $t$ for all $x\in\R^d$ and it is of class $C^l$ with respect to $x$ for almost all $t\geq 0$. Thus, $u_1$ is continuous and has a locally integrable derivative with respect to $t$ for all $x\in\R^d$ and $u_1$ is of class $C^m$ for all $t\geq 0$ and $\frac{\partial u_1}{\partial t}$ is of class $C^m$
for a.e. $t\geq 0$.   

We now apply the same procedure to $u_2(t,x)=(F(t,x)*_tf_r(t))*_{(t,x)}\frac{\partial f}{\partial t}(t,x)$, 
with $\cF_{x\rightarrow \xi}(f(t,x))$, $t\in[0,\infty)$, in place of $\widehat u_0(\xi)$.  Condition \eqref{eq:cond2} implies
\begin{align*}
	D_j^l&\frac{\partial}{\partial t}u_2(t,x)\\
	& =
	D_j^l\cF^{-1}_{\xi\rightarrow x}\left[f_r(t)*_t\left(\cL^{-1}_{s\to t}\left(\frac{s^r}{s^r+a(\xi)}-1\right)+\delta\right)
	\cdot
	\frac{\partial}{\partial t}\mathcal F_{x\rightarrow \xi}(f(t,x))\right],
\end{align*}
for all $j=1,...,d.$
This gives ($l\in\{1,...,m\}$)
\begin{align*}
&D_j^l\frac{\partial}{\partial t}u_2(t,x)
=
\mathcal F^{-1}_{\xi\rightarrow x}\left[\xi^l_j f_r(t)\cdot\frac{\partial}{\partial t}\cF_{x\rightarrow \xi}(f(t,x))\right]
\\
&+
 \sum_{p=1}^k (-1)^p f_{(p+1)r}(t)*_t\cF^{-1}_{\xi\rightarrow x}\left[\xi_j^la^{p}(\xi)\cF_{x\rightarrow \xi}\left(\frac{\partial f}{\partial t}(t,x)\right)\right]
\\
& \vspace{-30pt}+(-1)^{k+1}
f_r(t)*_t
\cF^{-1}_{\xi\rightarrow x}\left[\xi_j^l\cL^{-1}_{s\to t}\left(\frac{1}{s^{kr}(s^r+a(\xi))}\right)a^{k+1}(\xi)\cF_{x\rightarrow \xi}\left(\frac{\partial f}{\partial t}(t,x)\right)\right],
\end{align*}
for all $j=1,...,d.$
By the same arguments employed above we have that $u_2$ has the same regularity properties as $u_1$. The proof is complete.
\end{proof}
\begin{proof}[of Lemma \ref{lem:l2}]
Recall that $\cL_{t\to s}(\frac{\partial E}{\partial t}(t,\xi))=\frac{s^r}{s^r+a(\xi)}-1$, $\Re s>a(\xi)$, $\xi\in \R^d$ since $E(0,x)=1$. There holds, for $\Re s>0$, $\xi\in \R^d$,
\begin{align*}
	\cL_{t\to s}\left(\frac{\partial E}{\partial t}(t,\xi )\right)=\frac{s^r}{s^r+a(\xi)}-1&=\frac{-a(\xi)}{s^{r}+a(\xi)},\\
	\quad\frac{-a(\xi)}{s^{r}+a(\xi)}+\frac{a(\xi)}{s^r}&=\frac{a^2(\xi)}{s^{r}(s^r+a(\xi))},\\
	\frac{a^2(\xi)}{s^{r}(s^r+a(\xi))}-\frac{a^2(\xi)}{s^{2r}}&=\frac{-a^3(\xi)}{s^{2r}(s^r+a(\xi))}.
\end{align*}
After repeating this procedure $k$ times, we obtain that, for $\Re s>0$, 
\[
	\frac{s^r}{s^r+a(\xi)}-1=-\frac{a}{s^r}+\frac{a^2(\xi)}{s^{2r}}+...+(-1)^k
	\frac{a^k(\xi)}{s^{kr}}+(-1)^{k+1}\frac{a^{k+1}(\xi)}{s^{kr}(s+a(\xi))}, 
\]
where $k\in\N$ is determined so that $(k+1)r>1$. Note also that, for $\Re s>0$, $|s^r+a(\xi)|\geq |s|^r$. This implies
\[
	\frac{a^{k+1}(\xi)}{|s^{(k+1)r}+a(\xi)s^{kr}|}=
	\frac{a^{k+1}(\xi)}{|s|^{kr}\cdot|s^r+a(\xi)|}\leq \frac{a^{k+1}(\xi)}{|s|^{(k+1)r}}, \quad \Re s >0,\xi\in\R^d.
\]
We also note that 
\begin{align*}
\cL^{-1}_{s\to t}\left(\frac{1}{s^{pr}}\right)&=f_{pr}(t)\in L^1_{\loc}(0,\infty) \qquad \text{ for all } p=1,...,k, \\
\cL^{-1}_{s\to t}(\frac{1}{s^{(k+1)r}})&=f_{(k+1)r}(t) \text{ is continuous },
\end{align*}
and, by the assumption \eqref{eq:cond1}, $\cF^{-1}(a^p\widehat u_0)\in C(\R^d)$ for all $p=1,...,k+1.$ Together with \eqref{eq:cond0}, this completes the proof.
\end{proof}


\section{Variable coefficients equations}\label{sec:varcoeff}
In this section, 
we extend the results mentioned above to two classes of variable coefficients operators. The first case we consider is the one where $\Op(a)$
in an operator belonging to the so-called SG-calculus (see, e.g., Section 2 of \cite{CGP2} 
for a short summary of the main features of this calculus), 
with positive order. Namely, denoting, as usual, $\norm{y}=\sqrt{1+|y|^2}$, $y\in\R^d$, in the estimates below, the symbol $a$ is real-valued, 
and satisfies the following assumptions:

\begin{enumerate}[label=\textbf{(H\arabic*)},ref=\textbf{(H\arabic*)}]
\item \label{hyp:a_symbol} there exist $m,\mu\in(0,+\infty)$ such that $a\in S^{m,\mu}(\R^d\times \R^d)$;
\item \label{hyp:a_hypoellyptic} $a$ is non-negative, and
there exist $R>0$, $m^\prime\in [0,m]$ and  $\mu^\prime\in [0,\mu]$ such that, for any $(x,\xi)\in \R^{d}\times \R^d$ with $|x|+|\xi|\geq R$,
\begin{equation}
a(x,\xi)\geq C \<x\>^{m^\prime}\<\xi\>^{\mu^\prime},
\end{equation}
for some constant $C>0$ independent of $x$ and $\xi$;
\item \label{hyp:a_derivatives} for all multi-indices $\alpha, \beta\in \N^d$ there exist constants $C_{\alpha \beta}>0$ such that 
\begin{equation}
 \frac{\left|\partial_x^\alpha\partial_\xi^\beta a(x,\xi)\right|}{a(x,\xi)}\leq C_{\alpha \beta} \<x\>^{-|\alpha|}\<\xi\>^{-|\beta|},
\end{equation}
for any $(x,\xi)\in \R^{d}\times \R^d$ with $|x|+|\xi|\geq R$.
\end{enumerate}
We recall that, for $z,\zeta\in\R$, the corresponding so-called Sobolev-Kato (or \textit{weighted Sobolev}) space of order $(z,\zeta)$
is defined as
\begin{equation}\label{eq:skspace}
  	H^{z,\zeta}(\R^d)= \{u \in \scS^\prime(\R^{n}) \colon \|u\|_{z,\zeta}=
	\| \langle\cdot\rangle^z
	\langle D\rangle^\zeta u\|_{L^2}< \infty\}
\end{equation}
(see, e.g., Section 2 of \cite{CGP2} for some properties of these functional spaces and their interplay with the 
SG-calculus).

The following Theorem \ref{thm:main} is our second main result.
\begin{theorem}\label{thm:main}
	In the Cauchy problem \eqref{eq:CPmain}, assume $f=0$ and $u_0\in H^{\ell,\rho}(\R^d)$, see \eqref{eq:skspace}, and let $a$ satisfy assumptions 
	\ref{hyp:a_symbol}, \ref{hyp:a_hypoellyptic} and \ref{hyp:a_derivatives}. 
	Then the Cauchy problem \eqref{eq:CPmain} admits a unique solution 
	\[
		u\in C([0,+\infty),H^{\ell,\rho}(\R^d))\cap C((0,+\infty),H^{\ell+m^\prime,\rho+\mu^\prime}(\R^d)),
	\]
	given, modulo $C^\infty([0,+\infty),\scS(\R^d))$, by
	\begin{equation}\label{eq:reprform_u}
		u(t)=\Op(K_0(t))u_0,
	\end{equation}
	where 
	\begin{align}\label{eq:K_asexp_thm}
		K_0(t,x,\xi)&\sim \sum_{j\in \N} \frac{t^{jr}}{j!} A_j(x,\xi) E^{(j)}_{r,1}(-t^r a(x,\xi)), \quad t\in [0,\infty), x,\xi\in \R^d.
	\end{align}
	In \eqref{eq:K_asexp_thm}, $E_{r,1}$ denotes the Mittag-Leffler function, the symbols $A_j$, $j\in\N$, are defined in Proposition \ref{prop:cs_parametrix} below, 
	and it holds 
	\[
		K_0\in C([0,+\infty),S^{0,0}(\R^d\times\R^d)) \cap C((0,+\infty), S^{-m^\prime,-\mu^\prime}(\R^d\times\R^d)).
	\]
\end{theorem}

Our third main result, the subsequent Theorem \ref{thm:main_nhom}, deals with the non-homogeneous Cauchy problem \eqref{eq:CPmain}.
\begin{theorem}\label{thm:main_nhom}
In the Cauchy problem \eqref{eq:CPmain}, assume $u_0\in H^{\ell,\rho}(\R^d)$ and $f\in  C([0,\infty), H^{\ell-m^\prime+k,\rho-\mu^\prime+\kappa}(\R^d))$, 
$k\in(0,m^\prime]$, $\kappa\in(0,\mu^\prime]$,
and let $a$ satisfy assumptions \ref{hyp:a_symbol}, \ref{hyp:a_hypoellyptic} and \ref{hyp:a_derivatives}.
Then the Cauchy problem \eqref{eq:CPmain} admits a unique solution 
\[ u\in C([0,+\infty),H^{\ell,\rho}(\R^d)),\]
	given, modulo $C([0,+\infty),\scS(\R^d))$, by
	\begin{equation}\label{eq:reprform_u_nhom}
		u(t)=\Op(K_0(t))u_0+\int_0^t \Op(K_1(\tau))f(t-\tau)d\tau,
	\end{equation}
	where $K_0(t)$ is given by \eqref{eq:K_asexp_thm} and for all $ t\in [0,\infty), x,\xi\in \R^d,$ 
	\begin{equation}\label{eq:K1_asexp_thm}
		K_1(t,x,\xi)\sim \sum_{j\in \N} \frac{t^{jr+r-1}}{j!} A_j(x,\xi) E^{(j)}_{r,r}(-t^r a(x,\xi)),
	\end{equation}
	satisfies 
	\[
		K_1\in C([0,+\infty),S^{0,0}(\R^d\times\R^d)) \cap C((0,+\infty), S^{-pm^\prime,-p\mu^\prime}(\R^d\times\R^d)), \quad p\in(0,2].
	\]
	Moreover, if additionally $f\in  C((0,\infty), H^{\ell+\varepsilon,\rho+\varepsilon}(\R^d))$, $\varepsilon>0$ arbitrarily small,
	 the Cauchy problem \eqref{eq:CPmain} admits a unique solution 
	\[ u\in C([0,+\infty),H^{\ell,\rho}(\R^d))\cap C((0,+\infty),H^{\ell+m^\prime,\rho+\mu^\prime}(\R^d)),\]
	given, modulo $C([0,+\infty),\scS(\R^d))$, by \eqref{eq:reprform_u_nhom} as above.
\end{theorem}
We refer to  Section \ref{subs:M-L} 
for the Mittag-Leffler function $E_{r,r}$ and its derivatives.

\medskip

By a completely similar approach, we can prove an analogous results for a hypo-elliptic (pseudo)differential operator with symbol
$a=a(x,\xi)$ belonging to the (classical) H\"ormander calculus. Namely, consider the following alternative assumptions:
\begin{enumerate}[label=\textbf{(H\arabic*)$'$},ref=\textbf{(H\arabic*)$'$}]
\item \label{hyp:a_symbolprime} there exist $\mu\in(0,+\infty)$ such that $a\in S^{\mu}(\R^d\times \R^d)$;
\item \label{hyp:a_hypoellypticprime} a is non-negative, and there exist $R>0$ and  $\mu^\prime\in [0,\mu]$ such that, for any $(x,\xi)\in \R^{d}\times \R^d$ with $|\xi|\geq R$,
\[
	a(x,\xi)\geq C\<\xi\>^{\mu^\prime},
\]
for some constant $C>0$ independent of $x$ and $\xi$;
\item \label{hyp:a_derivativesprime} for all multi-indices $\alpha, \beta\in \N^d$ there exist constants $C_{\alpha \beta}>0$ such that 
\[
 \frac{	\left|\partial_x^\alpha\partial_\xi^\beta a(x,\xi)\right|}{a(x,\xi)}\leq C_{\alpha \beta} \<\xi\>^{-|\beta|},
\]
for any $(x,\xi)\in \R^{d}\times \R^d$ with $|\xi|\geq R$.\\
\end{enumerate}
The next Theorem \ref{thm:mainbis} is our fourth and final main result.
\begin{theorem}\label{thm:mainbis}
	\begin{enumerate}
	\item[ i)] In the Cauchy problem \eqref{eq:CPmain}, assume $f=0$ and $u_0\in H^{\rho}(\R^d)$, and let $a$ satisfy assumptions 
	\ref{hyp:a_symbolprime}, \ref{hyp:a_hypoellypticprime}, \ref{hyp:a_derivativesprime}. 
	%
	Then the Cauchy problem \eqref{eq:CPmain} admits a unique solution 
	\[
		u\in C([0,+\infty),H^{\rho}(\R^d))\cap C((0,+\infty),H^{\rho+\mu^\prime}(\R^d)),
	\]
	given, modulo $C^\infty([0,+\infty),C^\infty(\R^d))$, by \eqref{eq:reprform_u}, where $K_0$ is defined by \eqref{eq:K_asexp_thm}.
	In \eqref{eq:K_asexp_thm}, $E_{r,1}$ denotes the Mittag-Leffler function, the symbols $A_j$, $j\in\N$, are defined in Proposition \ref{prop:cs_parametrix} below, 
	and it holds 
	%
	%
	%
	\[
		K_0\in C([0,+\infty),S^{0}(\R^d\times\R^d)) \cap C((0,+\infty), S^{-\mu^\prime}(\R^d\times\R^d)).
	\]
	\item[ii)] In the Cauchy problem \eqref{eq:CPmain}, assume $u_0\in H^{\rho}(\R^d)$ and $f\in  C([0,\infty),$ $H^{\rho-\mu^\prime+\kappa}(\R^d))$,
	 $\kappa\in(0,\mu^\prime]$ 
	and let $a$ satisfy assumptions \ref{hyp:a_symbolprime}, \ref{hyp:a_hypoellypticprime}, \ref{hyp:a_derivativesprime}. 
	Then the Cauchy problem \eqref{eq:CPmain} admits a unique solution 
	\[ u\in C([0,+\infty),H^{\rho}(\R^d)),\]
	given, modulo $C^\infty([0,+\infty),C^\infty(\R^d))$, by
	\begin{equation}\label{eq:reprform_u_nhombis}
		u(t)=\Op(K_0(t))u_0+\int_0^t \Op(K_1(\tau))f(t-\tau)d\tau,
	\end{equation}
	where $K_0$ is given by \eqref{eq:K_asexp_thm}, $K_1$ is given by \eqref{eq:K1_asexp_thm} and satisfies 
	\[
		K_1\in C([0,+\infty),S^{0}(\R^d\times\R^d)) \cap C((0,+\infty), S^{-p\mu^\prime}(\R^d\times\R^d)), \quad p\in(0,2].
	\]
	Moreover, if additionally $f\in  C((0,\infty), H^{\rho+\varepsilon}(\R^d))$, $\varepsilon>0$ arbitrarily small,
	the Cauchy problem \eqref{eq:CPmain} admits a unique solution 
	\[ u\in C([0,+\infty),H^{\rho}(\R^d))\cap C((0,+\infty),H^{\rho+\mu^\prime}(\R^d)),\]
	given, modulo $C([0,+\infty),C^\infty(\R^d))$, by \eqref{eq:reprform_u_nhombis} as above.
	\end{enumerate}
\end{theorem}
\begin{remark}
	We remark that Theorems \ref{thm:main} and \ref{thm:main_nhom} hold true also in the setting $[0,+\infty)\times M$, involving weighted Sobolev
	spaces $H^{\ell,\rho}(M)$, where $M$ is a so-called SG-manifold, a manifold with cylindrical ends, 
	or (the interior of) an asymptotically Euclidean manifold (see, e.g., \cite[Appendix]{CD21} and \cite{Cord,ME,Schro}). 
	Analogously, Theorem \ref{thm:mainbis} holds true in the setting $[0,+\infty)\times M$, involving Sobolev spaces $H^{\rho}(M)$,
	where $M$ is a closed manifold.
	For the sake of brevity, in the sequel we give the detailed proof of Theorems \ref{thm:main} and \ref{thm:main_nhom} only, 
	omitting the proofs of Theorem \ref{thm:mainbis},
	which is completely similar, as well as the statements and the proofs of the results on manifolds, which follow by the main
	results stated in this section, reducing to the setting $[0,+\infty)\times\R^d$, by means of the usual approach based on
	local charts and subordinate partition of unity, compatible with the geometric setting and employed symbolic structures.
\end{remark}

As a first step in proving the results stated above, we apply the Laplace transform $\cL$ with respect to $t$ in \eqref{eq:CPmain}.
Since $\Op(a)$ and $\cL$ commute (see Section 4 of \cite{CGP2}), 
we derive that $U(s,x):=(\mathcal{L}(u(\cdot,x)))(s)$ solves the parameter-dependent 
pseudodifferential equation 
\begin{equation}
\label{eq:equation_Laplace_transform}
\Big(\Big(s^r + \Op(a)\Big)U(s,\cdot)\Big)(x)=s^{r-1}u_0(x)+F(s,x), \quad x\in\R^d,
\end{equation}
where $F(s,x):=(\mathcal{L}f(\cdot,x))(s)$, for every $s\in \C$ with $\Re s>\lambda$ sufficiently large. Here, for all $s=|s| e^{i \theta}\in\C$, we are denoting by $s^r+\Op(a)$ the pseudodifferential operator with symbol $s^r +a(x,\xi)$, where $s^r$ denotes the complex root of order $r$ on the principal branch, that is, 
$s^r:= |s|^r e^{i r \theta}$ with $\theta \in (-\pi,\pi)$. In the sequel, we will often write $U(s)$ for $U(s)\colon x\mapsto U(s,x)$, $F(s)$ for $F(s)\colon x\mapsto F(s,x)$, 
and analogous notation for functions or distributions on $\R^d$ depending on the parameter $s\in\C$. The subsequent Proposition \ref{prop:cs_parametrix}
and Corollary \ref{cor:rs} are the key technical results that we need here.
\begin{proposition}\label{prop:cs_parametrix}
There exists a family of symbols $c_s\in S^{-m^\prime,-\mu^\prime}(\R^d\times \R^d)$ such that, for any $s\in \C_\lambda=\{s\in\C\colon \Re s>\lambda>0\}$,
$\lambda$ sufficiently large, $\Op(c_s)$ is a parametrix of $\Op(b_s)$, that is,
\begin{equation}\label{eq:cs_par}
	\Op(b_s)\Op(c_s)=I+\Op(r_{1s}), \quad \Op(c_s)\Op(b_s)=I+\Op(r_{2s}),
\end{equation}
for suitable $r_{1s}, r_{2s}\in S^{-\infty, -\infty}(\R^d\times \R^d)$. More precisely, there exist symbols 
$Q_j\in S^{(j+2)m-j-1, (j+2)\mu-j-1}$, $j\in \N$, independent of $s$, such that $c_s$ is given by the asymptotic sum
\begin{equation}
\label{eq:cs_asymptoticsum}
c_s(x,\xi)\sim \frac{1}{s^r+a(x,\xi)} + \sum_{j\in\N} \frac{Q_j(x,\xi)}{[s^r+a(x,\xi)]^{j+3}} = \sum_{j\in\N}\frac{A_j(x,\xi)}{[s^r+a(x,\xi)]^{j+1}},
\end{equation}
$A_0\equiv 1$, $A_1\equiv0$, $A_j=Q_{j-2} \in S^{jm-j+1,j\mu-j+1}$, $j\ge2$, and satisfies, for any $k\in \N$, the estimates
\[ 
	||| c_s|||^{-m^\prime,-\mu^\prime}_k\leq C_k,
\]
for suitable constants $C_k>0$ independent of $s\in\C_\lambda$ and the SG-seminorms $|||\cdot|||^{\ell,\rho}_k$, $\ell,\rho\in\R$, $k\in\N$
(see Remark \ref{rem:sgseminorms} below). In particular, for any $j\in \N$, $j\ge2$, the symbol $A_j$ admits an asymptotic expansion 
of the form
\begin{equation}
\label{eq:asymptotic_Aj} 
A_j\sim \sum_{|\theta|>j-2}\tilde{P}_{j}^{\theta\theta},
\end{equation}
where
\begin{equation*}
\tilde{P}_{j}^{\theta\theta}\in \vspan\left\{\partial^{\theta_1}_x\partial^{\sigma_1}_\xi a(x,\xi)\cdots\partial^{\theta_j}_x\partial^{\sigma_j}_\xi a(x,\xi)
		\colon \sum_{k=1}^j\theta_k=\sum_{k=1}^j\sigma_k=\theta \right\}.
\end{equation*}
\end{proposition}
\begin{remark}\label{rem:sgseminorms}
In Proposition \ref{prop:cs_parametrix}, for any $\ell, \rho\in \R$ and $p\in S^{\ell, \rho}(\R^d\times \R^d)$, we are considering the family of seminorms
\[ |||p|||^{\ell,\rho}_k=\sup_{|\alpha|+|\beta|\leq k}\sup_{(x,y)\in \R^{2d}}|\partial_x^\alpha \partial_\xi^\beta p(x,\xi)|\<x\>^{-\ell+|\alpha|}\<\xi\>^{-\rho+|\beta|},\]
with $k\in \N$, which defines a Fr\'echet topology on $S^{\ell,\rho}(\R^d\times \R^d)$ (notice that these seminorms are actually norms; see, e.g.,
\cite[Section 2]{CGP2} and \cite[Ch. 1]{Cord}).
\end{remark}
\begin{cor}\label{cor:rs}
	The symbols $r_{1s}$ and $r_{2s}$ of the remainders in \eqref{eq:cs_par} satisfy 
	\begin{align*}
		\forall M\in\N\;\forall\alpha,\beta\in\N^d\; &\exists C_{\alpha\beta}>0 \; \forall x,\xi\in\R^d\;\forall s\in\C_\lambda\;\;
		\\
		&
		|\partial^\alpha_\xi\partial^\beta_x r_{1s}(x,\xi)|\le |s|^{-rM}\csi^{-M-|\alpha|}\x^{-M-|\beta|}
		\\
		\text{ and }\;
		&|\partial^\alpha_\xi\partial^\beta_x r_{2s}(x,\xi)|\le |s|^{-rM}\csi^{-M-|\alpha|}\x^{-M-|\beta|},
	\end{align*}
	with $\lambda>0$ sufficiently large.
	We write $r_{1s},r_{2s} \in |s|^{-\infty}S^{-\infty,-\infty}(\R^d\times\R^d)$. It follows that the corresponding kernels $k_{1}(s,x,y)$ and $k_2(s,x,y)$
	of $\Op(r_{1s})$ and $\Op(r_{2s})$, respectively, satisfy analogous estimates in $(s,x,y)\in\C_\lambda\times\R^d\times\R^d$, and are analytic
	functions on $\C_\lambda$ taking values in $\scS(\R^d\times\R^d)$, with $\lambda>0$ sufficiently large.
\end{cor}
The detailed proof of Proposition \ref{prop:cs_parametrix} and Corollary \ref{cor:rs} are given in Section 3 of \cite{CGP2}. 
By means of the parametrix $\Op(c_s)$ obtained there, we are now able to get the representation of the solution $u=u(t,x)$ 
to \eqref{eq:CPmain} in terms of Mittag-Leffler functions (see Section \ref{subs:M-L}) 
claimed in Theorem \ref{thm:main}. 
%
%
%
\begin{proof}[of Theorem \ref{thm:main}]
Let us first prove uniqueness. Let 
\[
	u_1,u_2\in C([0,+\infty),H^{\ell,\rho})\cap C((0,+\infty),H^{\ell+m^\prime,\rho+\mu^\prime})
\]
be two solutions of \eqref{eq:CPmain}. \\
Applying the Laplace transform, and denoting $U_j(s,x)=\cL_{t\to s}(u_j(t,x))$, $j=1,2$, both $U_1$ and $U_2$
satisfy \eqref{eq:equation_Laplace_transform} on some half-plane $\C_\lambda$. Setting $W=U_1-U_2$, it follows 
\[
	(s^r+\Op(a))W(s,x)=0\Leftrightarrow \Op(a)W(s)=(-s^r)W(s),
\]
that is, $W(s)$ is an eigenvector of the (closable) linear operator $\Op(a)\colon\scS\subset L^2\to L^2$, associated with the eigenvalue $\kappa=-s^r$, $\Re s>\lambda$.
By adapting \cite[Ex. 4]{Epstein} (cf. also \cite[p. 237--238]{AgrMar89}), taking into account that $a(x,\xi)>0$, it follows that 
\begin{equation}\label{eq:eigenvhalfplane}
	\Op(a)W=\kappa W\Rightarrow \Re\kappa\ge -c,
\end{equation}
for some constant $c\in\R$. Indeed, $H^{\ell,\rho}$, $\ell,\rho\in\R$, is a family of interpolation spaces 
(see \cite{Epstein}). Moreover, we can write $\Op(a)=\Op^w(a)+\Op(a_1)$, $a_1\in S^{m-1,\mu-1}$, and
\begin{align*}
	\Op(a)=\underbrace{\frac{\Op(a)+\Op(a)^*}{2}}_{=\Op(a_0)
	\text{, selfadjoint part of $\Op(a)$}}&\mod\Op(S^{m-1,\mu-1})\\
& =\Op^w(a)\mod\Op(S^{m-1,\mu-1}).
\end{align*}
Notice that also $\Op^w(a)$ is selfadjoint, since $a$ is real-valued (see, e.g., \cite[Prop. 1.2.11]{NiRo}). Since $\Op^w(a)$ is bounded from below 
(see, e.g., \cite[Lemma 4.2.9]{NiRo}), by adding a suitable constant $K$, it becomes nonnegative. Then, we can apply \cite[Theorem 2]{Epstein}, 
with $L=\Op^w(a)+K$, $T=\Op(a_1)-K$, and obtain \eqref{eq:eigenvhalfplane}. Since 
\[
	\Re s>\left(\dfrac{c}{\cos\frac{r\pi}{2}}\right)^\frac{1}{r}\Rightarrow\Re (-s^r)<-c,
\]
taking $s\in\C_\lambda$ with $\lambda$ large enough implies $W\equiv0\Leftrightarrow U_1\equiv U_2$ on $\C_\lambda$. By inverse Laplace transform
(see, e.g., \cite[Theorem 4.14]{CGP2}), 
we conclude $u_1\equiv u_2$, as claimed.

Concerning existence and the representation formula \eqref{eq:reprform_u}, with $u\in C((0,\infty),\scSp)$,
applying $\Op(c_s)$ to both sides of \eqref{eq:equation_Laplace_transform} with $F=0$, we find
\begin{equation}
\label{eq:U(s)}
U(s)=\Op\big(s^{r-1}c_s\big)u_0-\Op\big(r_{2s}\big)U(s).
\end{equation}
As a consequence, we may write 
\begin{equation}\label{eq:solhom}
	u(t)=\mathcal{L}_{s\to t}^{-1}\big(\Op\big(s^{r-1}c_s)u_0)+\mathcal{L}_{s\to t}^{-1}\big(\Op\big(r_{2s}) U(s)).
\end{equation}

Set
\[
	G(s,x)=[\Op(r_{2s})U(s)](x)=[U(s)](k_2(s,x,\cdot)),
\]
where $k_2(s)\in\scS$ is the Schwartz kernel of $\Op(r_{2s})$, which,
by Corollary \ref{cor:rs}, is rapidly decaying with respect to $|s|$ and analytic with respect to $s\in\C_\lambda$. Then, $G(s)\in\scS$,
and is rapidly decaying with respect to $|s|$ and analytic with respect to $s\in\C_\lambda$ as well. 
By the properties of the inverse Laplace transform (see, e.g., \cite[Theorem 4.14]{CGP2}), 
it follows that $\cL^{-1}_{s\to t}(G(s))\in C^\infty([0,+\infty)_t,\scS)$.

In view of \eqref{eq:laplace_tranform_MIttagLeffler}, we also obtain
\begin{align*}
\mathcal{L}_{s\to t}^{-1}\big(\big(\Op\big(s^{r-1}c_s\big)\big)u_0&= \Op\Big(\mathcal{L}^{-1}_{s\to t}\big(s^{r-1}c_s\big)\Big)u_0\\
&\sim \Op\bigg(\mathcal{L}^{-1}_{s\to t}\bigg(\sum_{j\in \N}s^{r-1}c_{js}\bigg)\bigg)u_0\\
&= \Op\bigg(\sum_{j\in \N}A_j \, \mathcal{L}^{-1}_{s\to t}\bigg(\frac{s^{r-1}}{(s^r+a)^{j+1}}\bigg)\bigg)u_0\\
&= \Op(K(t))u_0,
\end{align*}
where 
\[ 
	K_0(t,x,\xi) \sim \sum_{j\in \N} \frac{t^{jr}}{j!} A_j(x,\xi) E^{(j)}_{r,1}(-t^r a(x,\xi)).
\]
It is immediate to see that $K_0(0)\in S^{0,0}$.
The claim then follows by the subsequent Lemma \ref{lem:Kt}, where we show that, for any $t>0$, $K_0(t)=K_0(t,x,\xi)$ is a symbol in $S^{-m^\prime,-\mu^\prime}$.
\end{proof}
\begin{lemma}\label{lem:Kt}
The family of symbols $K_0(t)$, $t\in[0,+\infty)$, satisfies 
\[
	K_0\in C([0,\infty),S^{0,0}(\R^d\times\R^d))\cap C((0,\infty), S^{-m^\prime,-\mu^\prime}(\R^d\times\R^d)).
\]

\end{lemma}
\begin{proof}
Employing Lemma \ref{lem:integral_representation_MittagLeffler} we can write
\begin{align}
E_{r,1}(-t^r a(x,\xi))= F_{r,1}(-t^ra(x,\xi)),
\end{align}
where 
\begin{equation}
\label{eq:Fr1}
\begin{aligned} 
&F_{r,1}(-t^ra(x,\xi))\\
& =\frac{a(x,\xi)\sin(r\pi)}{\pi r
 }\int_0^\infty e^{-\tau^\frac{1}{r}}\frac{t^r }{\tau^2+2\tau t^r a(x,\xi)\cos(r\pi )+t^{2r}a(x,\xi)^2}\,d\tau.
\end{aligned}
\end{equation}

Then, we have proved
\begin{equation}\label{eq:Kasexp}
	K_0(t,x,\xi) \sim \sum_{j\geq 0} \frac{t^{jr}}{j!} A_j(x,\xi) F^{(j)}_{r,1}(-t^r a(x,\xi)),
\end{equation} 
where $F_{r,1}$ is defined by \eqref{eq:Fr1}, 
and, for  any $j\geq 1$, the function $F^{(j)}_{r,1}(-t^r a(x,\xi))$ is the derivative of order $j$ of $F_{r,1}$ evaluated in $-t^r a(x,\xi)$.
We now show that \eqref{eq:Kasexp} indeed provides an asymptotic expansion for $K(t)$, which will conclude the proof of our claim.
In fact:
\begin{itemize}
	\item[-] the computations below also show that all the terms in the expansion \eqref{eq:Kasexp} have the desired continuity properties,
	with respect to $t\in[0,\infty)$, respectively $t\in(0,+\infty)$, to the corresponding symbol spaces;
	\item[-] by standard arguments (see, e.g., \cite[\S\ 2.1]{StRay}), such $t$-continuity properties extend from the elements 
	of the expansion to the asymptotic sum.
\end{itemize}
 
%
\smallskip 

Set $Q_R=\{(x,\xi)\in\R^d\times\R^d\colon |x|+|\xi|\le R\}$, $R>0$, and choose a cut-off function $\chi\in C_0^\infty(\R^d\times\R^d)$
such that $0\le\chi\le1$, $\supp\chi\subseteq Q_{2R}$, and $\chi|_{Q_{3R/2}}\equiv1$. Set also
\[
	K_{jc}(t,x,\xi)=\frac{t^{jr}}{j!} A_j(x,\xi) F^{(j)}_{r,1}(-t^r a(x,\xi))\chi(x,\xi)
\] and
\[
	K_{j\infty}(t,x,\xi)=\frac{t^{jr}}{j!} A_j(x,\xi) F^{(j)}_{r,1}(-t^r a(x,\xi))(1-\chi(x,\xi)),
\]
so that
\[
	K_0(t,x,\xi)\sim\sum_{j\geq 0} K_{jc}(t,x,\xi) + \sum_{j\ge 0} K_{j\infty}(t,x,\xi)].
\]
Since, for any $j\in\N$, $K_{jc}\in C([0,\infty), C_0^{\infty})\subset C([0,\infty), S^{-\infty,-\infty})$, it follows
\[
	\sum_{j\geq 0} K_{jc}(t,x,\xi)\in C([0,\infty), S^{-\infty,-\infty})\subset 
	C([0,\infty),S^{0,0})\cap C((0,\infty), S^{-m^\prime,-\mu^\prime}).
\]
So, we need only to show that the $K_{j\infty}$, $j\in\N$, provide an SG-asymptotic sum of the claimed order for $t\in[0,+\infty)$ and
$t\in(0,+\infty)$, respectively. In particular, since for any $\alpha,\beta\in\N^d$
such that $|\alpha|+|\beta|\ge1$ we have $\partial^\alpha_x\partial^\beta_\xi(1-\chi(x,\xi))  \in C_0^\infty\subset S^{-\infty,-\infty}$,
it is enough to estimate the derivatives only of the other factors on the support of $1-\chi(x,\xi)$, that is, for $|x|+|\xi|\ge R$.
Applying Fa\`a di Bruno' formula, we know that for any $\alpha, \beta\in \Z_+^d$ with $|\alpha|+|\beta|\geq 1$ it holds
\begin{equation*}
\begin{aligned}
&\partial_x^\alpha\partial_\xi^\beta \left(F_{r,1}^{(j)}(-t^ra(x,\xi)\right) \\
& = \!\!\!\!
\sum_{\ell=1}^{|\alpha|+|\beta|}  t^{\ell r}F^{(j+\ell)}_{r,1}\left(-t^r a(x,\xi)\right) \!\!\!\!\!\sum_{\substack{\alpha_1+\dots +\alpha_\ell=\alpha\\ \beta_1+\dots +\beta_\ell=\beta}} C^{\beta_1,\dots,\beta_\ell}_{\alpha_1, \dots, \alpha_\ell}\partial_{x}^{\alpha_1} \partial_\xi^{\beta_1}a(x,\xi)\cdots \partial_{x}^{\alpha_\ell} \partial_\xi^{\beta_\ell}a(x,\xi)
\end{aligned}
\end{equation*}
for suitable constants $C^{\beta_1,\dots,\beta_\ell}_{\alpha_1, \dots, \alpha_\ell}\in \R$. Then, applying Lemma \ref{lem:Falpha,beta} and assumptions \ref{hyp:a_hypoellyptic}, \ref{hyp:a_derivatives}, on $\supp(1-\chi(x,\xi))$ we may estimate
\begin{equation*}
\begin{aligned}
& \left|\partial_x^\alpha \partial_\xi^\beta  \left(F_{r,1}^{(j)}(-t^ra(x,\xi)\right)\right| \\
& \lesssim \sum_{\ell=1}^{|\alpha|+|\beta|} \!\!\! \frac{ 1}{(1+t^ra(x,\xi)) t^{jr}a(x,\xi)^j}   \!\!\!\sum_{\substack{\alpha_1+\dots +\alpha_\ell=\alpha\\ \beta_1+\dots +\beta_\ell=\beta}} \frac{|\partial_{x}^{\alpha_1} \partial_\xi^{\beta_1}a(x,\xi)|}{a(x,\xi)}\dots \frac{|\partial_{x}^{\alpha_\ell} \partial_\xi^{\beta_\ell}a(x,\xi)|}{a(x,\xi)}  \\
& 	\lesssim  \sum_{\ell=1}^{|\alpha|+|\beta|}\frac{t^{-jr}a(x,\xi)^{-j}}{1+t^r a(x,\xi)}  \sum_{k=\ell}^{|\alpha|+|\beta|}\sum_{\substack{\alpha_1+\dots +\alpha_k=\alpha\\ \beta_1+\dots +\beta_k=\beta}}\<x\>^{-|\alpha_1|} \<\xi\>^{-|\beta_1|}\dots \<x\>^{-|\alpha_\ell|} \<\xi\>^{-|\beta_\ell|}\\
& \lesssim\frac{t^{-jr}a(x,\xi)^{-j}}{1+t^r a(x,\xi)} \<x\>^{-|\alpha|}\<\xi\>^{-|\beta|}.
\end{aligned}
\end{equation*}

Now, let us fix $N>j\max\{m-m^\prime+1,\mu-\mu^\prime+1\}-1$. Employing the asymptotic expansion \eqref{eq:asymptotic_Aj} of $A_j$, 
for any $\alpha,\beta \in \Z_+^d$ and $j\geq 2$ we can write
\begin{equation*}
 \partial_x^\alpha \partial_\xi^\beta  A_j=  \sum_{|\theta|=j-1}^N  \partial_x^\alpha \partial_\xi^\beta \tilde{P}_{j}^{\theta\theta}+ \partial_x^\alpha \partial_\xi^\beta \tilde R_{N},
\end{equation*}
where $ \partial_x^\alpha \partial_\xi^\beta \tilde{P}_{j}^{\theta\theta}$ belongs to
\[ \vspan\left\{\partial^{\theta_1}_x\partial^{\sigma_1}_\xi a(x,\xi)\cdots\partial^{\theta_j}_x\partial^{\sigma_j}_\xi a(x,\xi)
		\colon \sum_{k=1}^j\theta_k=\theta+\alpha, \; \sum_{k=1}^j\sigma_k=\theta+\beta \right\},\]
and $\partial_x^\alpha \partial_\xi^\beta \tilde{R}_N\in S^{jm-N-|\alpha|,j\mu-N-|\beta|}$. 
Our choice of $N$, together with assumption \ref{hyp:a_derivatives}, allows to estimate, on $\supp(1-\chi(x,\xi))$,
\begin{align*}
|t^{jr}&\partial_x^\alpha \partial_\xi^\beta \big(A_j(x,\xi)F_{r,1}^{(j)}(-t^r a(x,\xi))\big)|\\
& \lesssim t^{jr} \sum_{\substack{\alpha_1+\alpha_2=\alpha\\\beta_1+\beta_2=\beta}} |\partial_x^{\alpha_1}\partial_\xi^{\beta_1}  A_j(x,\xi) ||\partial_x^{\alpha_2}\partial_\xi^{\beta_2}  (F_{r,1}^{(j)}(-t^r a(x,\xi)))|\lesssim\frac{ 1}{1+t^r a(x,\xi)} \\
& \times  \sum_{\substack{\alpha_1+\alpha_2=\alpha\\\beta_1+\beta_2=\beta}} \<x\>^{-|\alpha_2|}\<\xi\>^{-|\beta_2|} 
\bigg(\sum_{|\theta|=j-1}^N 
\sum_{\substack{\theta_1+\dots +\theta_j=\theta+\alpha_1\\ \sigma_1+\dots +\sigma_j=\theta+\beta_1}} 
\frac{|\partial_{x}^{\theta_1} \partial_\xi^{\sigma_1}a(x,\xi)|}{ a(x,\xi)}  \cdots \\
 & \hspace{150pt} \cdots \frac{|\partial_{x}^{\theta_j} \partial_\xi^{\sigma_j}a(x,\xi)|}{a(x,\xi)} 
+\frac{|\partial_x^{\alpha_1} \partial_\xi^{\beta_1} R_N(x,\xi)|}{a(x,\xi)^j}\bigg)\\
& \lesssim \frac{\<x\>^{-j+1-|\alpha|}\<\xi\>^{-j+1-|\beta|}}{1+t^r a(x,\xi)},
\end{align*}
for any $t\geq 0$ and $j\geq 2$ (recall that $A_0\equiv 1$ and $A_1\equiv 0$). 
This proves our claim that \eqref{eq:Kasexp} provides an asymptotic expansion in the SG-calculus for $K(t)$. 
Indeed, for any $\alpha, \beta\in \N^d$ and $t\geq 0$, we obtain, recalling that $a$ is non-negative,
\begin{align*}
|\partial_x^\alpha\partial_\xi^\beta K_{0\infty}(t,x,\xi)| &\le C_{0\alpha\beta} \frac{\<x\>^{-|\alpha|}\<\xi\>^{-|\beta|}}{1+t^ra(x,\xi)}
\le C_{0\alpha\beta} \<x\>^{-|\alpha|}\<\xi\>^{-|\beta|},
\\
|\partial_x^\alpha\partial_\xi^\beta K_{j\infty}(t,x,\xi)| &\le C_{j\alpha\beta} \frac{\<x\>^{1-j-|\alpha|}\<\xi\>^{1-j-|\beta|}}{1+t^ra(x,\xi)}
\le C_{j\alpha\beta} \<x\>^{1-j-|\alpha|}\<\xi\>^{1-j-|\beta|}, 
\end{align*}
for any $j\ge2,$ for some constants $C_{j\alpha\beta}>0$, $j\in\N\setminus\{1\}$, independent of $t\geq 0$, and arbitrary $x,\xi\in \R^d$. 
We conclude that, for any $t\ge0$,  it holds
\begin{equation*}
	K_\infty(t,x,\xi)\sim\sum_{j\ge0}K_{j\infty}(t,x,\xi) \in S^{0,0}.
\end{equation*}
Moreover, for any $t>0$ arbitrarily small, in view of assumption \ref{hyp:a_hypoellyptic}, for any $\alpha, \beta\in \N^d$ we may likewise estimate
\begin{align*}
|\partial_x^\alpha\partial_\xi^\beta K_{0\infty}(t,x,\xi)| &\le C_{0\alpha\beta} \frac{\<x\>^{-|\alpha|}\<\xi\>^{-|\beta|}}{1+t^ra(x,\xi)}
\le t^{-r} C_{0\alpha\beta} \<x\>^{-m^\prime-|\alpha|}\<\xi\>^{-\mu^\prime-|\beta|},
\\
|\partial_x^\alpha\partial_\xi^\beta K_{j\infty}(t,x,\xi)| &\le C_{j\alpha\beta} \frac{\<x\>^{1-j-|\alpha|}\<\xi\>^{1-j-|\beta|}}{1+t^ra(x,\xi)}
\\ & \le t^{-r} C_{j\alpha\beta} \<x\>^{-m^\prime+1-j-|\alpha|}\<\xi\>^{-\mu^\prime+1-j-|\beta|}, \quad j\ge2,
\end{align*}
and conclude that, for any $t>0$, it holds 
\begin{equation*}
	K_\infty(t,x,\xi)\sim\sum_{j\ge0}K_{j\infty}(t,x,\xi) \in t^{-r}S^{-m^\prime,-\mu^\prime}.
\end{equation*}
The proof is complete.
\end{proof}

We now prove our third main result, about the non-homogeneous Cauchy problem.
\begin{proof}[of Theorem \ref{thm:main_nhom}]
The uniqueness claim follows by an argument completely analogous to the one given to prove uniqueness of the solution 
for the homogeneous case $f\equiv 0$, in the proof of Theorem \ref{thm:main} above.

To prove existence, as in the proof of Theorem \ref{thm:main}, with $u\in C((0,+\infty),\scSp)$, we 
apply $\Op(c_s)$ to both sides of \eqref{eq:equation_Laplace_transform}, obtaining
\begin{equation}
\label{eq:U(s)-nhom}
U(s)=\Op\big(s^{r-1}c_s\big)u_0-\Op\big(r_{2s}\big)U(s)+\Op(c_s)F(s),
\end{equation}
and, by inverse Laplace transform, we may then write 
\begin{equation}\label{eq:solnonhom}
	u(t)=\mathcal{L}_{s\to t}^{-1}\big(\Op\big(s^{r-1}c_s)u_0)+\mathcal{L}_{s\to t}^{-1}\big(\Op\big(r_{2s}) U(s)\big) +
	\mathcal{L}_{s\to t}^{-1}\big(\Op(c_s)F(s)\big).
\end{equation}
In \eqref{eq:solnonhom}, the first two summands are identical to those in \eqref{eq:solhom}. For the third term, 
recalling that $\Op(\cdot)$ and $\cL^{-1}$ commute 
(see \cite[Remark 4.18]{CGP2}), we obtain
%
\begin{align*}
[\cL^{-1}_{s\to t}(\Op(c_s)F(s))](x)& = \frac{1}{(2\pi)^d}\int_{\R^d} e^{ix\cdot \xi} \cL^{-1}_{s\to t}(c_s(x,\xi))\ast_{t} \widehat{f}(t, \xi)d\xi  \\
& =  \frac{1}{(2\pi)^d}\int_0^t \int_{\R^d} e^{ix\cdot \xi} \cL^{-1}_{s\to \tau}(c_s(x, \xi))\hat{f}(t-\tau,\xi) d\xi d\tau\\
& = \int_0^t  [\Op(\cL^{-1}_{s\to \tau}(c_s))f(t-\tau)](x) d\tau.
\end{align*}
In particular, by \eqref{eq:cs_asymptoticsum}, in view of \eqref{eq:laplace_tranform_MIttagLeffler} we find
\begin{align*}
K_1(t,x,\xi):= \mathcal{L}_{s\to t}^{-1}(c_s(x,\xi))&\sim \cL^{-1}_{s\to t}\bigg(\sum_{j\in \N}c_{js}(x,\xi)\bigg)\\
&=\sum_{j\in \N}A_j (x,\xi)\, \mathcal{L}^{-1}_{s\to t}\bigg(\frac{1}{(s^r+a(x,\xi))^{j+1}}\bigg) \\
& = \sum_{j\in \N} \frac{t^{jr+r-1}}{j!}A_j(x,\xi) E_{r,r}^{(j)}(-t^r a(x,\xi)).
\end{align*}
The desired claims follow by the subsequent Lemma \ref{lem:K1t}, which shows that
$K_1(t)$ is, uniformly with respect to $t\ge0$, a symbol in $S^{0,0}$, while, for $t>0$,
$K_1(t)$ is a symbol in $t^{-1+(1-p)r}\cdot S^{-pm^\prime,-p\mu^\prime}$, $p\in(0,2]$. Indeed,
the first two terms are in $H^{\ell+m^\prime,\rho+\mu^\prime}$ and $\scS\hookrightarrow H^{\ell+m^\prime,\rho+\mu^\prime}$, respectively, for $t>0$,
and continuously depending on $t$, so that they also belong to $C([0,+\infty),H^{\ell,\rho})$.
For the third term, since $k,\kappa>0$, choose, as it is possible, $p\in[0,1)$ such that
\begin{align*}
	\begin{cases}
		\ell - m^\prime + k + p m^\prime \ge \ell
		\\
		\rho - \mu^\prime + \kappa + p \mu^\prime \ge \rho
	\end{cases}
	&\Leftrightarrow
	\begin{cases}
		m^\prime p \ge m^\prime - k
		\\
		\mu^\prime p \ge \mu^\prime - \kappa 
	\end{cases} \\
&	\Leftrightarrow
	\begin{cases}
		\text{any $p$ if $m^\prime=0$} \vee p \ge 1-\dfrac{k}{m^\prime} \text{ if $m^\prime>0$}
		\\
		\text{any $p$ if $\mu^\prime=0$} \vee p \ge 1-\dfrac{\kappa}{\mu^\prime} \text{ if $\mu^\prime>0$}.\rule{0mm}{7mm}
	\end{cases}
\end{align*}
Then, by the above computations, in view of Lemma \ref{lem:K1t} and the immersion between the Sobolev-Kato spaces, it holds,
for a constant $B>0$ depending only on $\ell,\rho,m^\prime,\mu^\prime,k,\kappa,d$, a constant $K>0$, depending only on a finite
number of seminorms of $K_1(t)$ (see, e.g., \cite[Theorem 2.1]{CGP2} and \cite[Ch. 3]{Cord}) 
and any $t\ge0$,
\begin{align*}
&	\left\|
	\int_0^t  \Op(\cL^{-1}_{s\to \tau}(c_s)) f(t-\tau) d\tau
	\right\|_{H^{\ell,\rho}}\\
& \le B
	\int_0^t  \left\|\Op(\cL^{-1}_{s\to \tau}(c_s))f(t-\tau)\right\|_{H^{\ell-m^\prime+k+p m^\prime,\rho-\mu^\prime+\kappa+p \mu^\prime}} d\tau
	\\
	&\le 
	B \int_0^t
	\|\Op(\cL^{-1}_{s\to \tau}(c_s))\|_{\cL(H^{\ell-m^\prime+k,\rho-\mu^\prime+\kappa}, H^{\ell-m^\prime+k+p m^\prime,\rho-\mu^\prime+\kappa+p \mu^\prime})}\\ & 
	\hspace{170pt} \times
	\|f(t-\tau)\|_{H^{\ell-m^\prime+k,\rho-\mu^\prime+\kappa}} \,d\tau
	\\
&	\le 
	B K \left(\max_{\tau\in[0,t]}\|f(\tau)\|_{H^{\ell-m^\prime+k,\rho-\mu^\prime+\kappa}}\right)\int_0^t \ta^{-1+(1-p)r}\,d\tau <\infty.
\end{align*}
We conclude $u\in C([0,+\infty), H^{\ell,\rho})$, as claimed. Similarly, if additionally $f\in C([0,+\infty),H^{\ell+\varepsilon,\rho+\varepsilon})$, $\varepsilon>0$
arbitrarily small, choose $p\in(0,1)$ such that
\begin{align*}
	\begin{cases}
		\ell + \varepsilon + p m^\prime \ge \ell + m^\prime
		\\
		\rho + \varepsilon + p \mu^\prime \ge \rho + \mu^\prime
	\end{cases}
&	\Leftrightarrow
	\begin{cases}
		m^\prime p \ge m^\prime - \varepsilon
		\\
		\mu^\prime p \ge \mu^\prime - \varepsilon 
	\end{cases}\\
&
	\Leftrightarrow
	\begin{cases}
		\text{any $p$ if $m^\prime=0$} \vee p \ge 1-\dfrac{\varepsilon}{m^\prime} \text{ if $m^\prime>0$}
		\\
		\text{any $p$ if $\mu^\prime=0$} \vee p \ge 1-\dfrac{\varepsilon}{\mu^\prime} \text{ if $\mu^\prime>0$}.\rule{0mm}{7mm}
	\end{cases}
\end{align*}
Then, as above, for suitable constants $B,K>0$ and any $t\ge0$,
\begin{align*}
	&\left\|
	\int_0^t  \Op(\cL^{-1}_{s\to \tau}(c_s))f(t-\tau) d\tau
	\right\|_{H^{\ell+m^\prime,\rho+\mu^\prime}}\\
&	\le B
	\int_0^t  \left\|\Op(\cL^{-1}_{s\to \tau}(c_s))f(t-\tau)\right\|_{H^{\ell+\varepsilon+p m^\prime,\rho+\varepsilon+\mu^\prime}} d\tau
	\\
&	\le 
	B \int_0^t
	\|\Op(\cL^{-1}_{s\to \tau}(c_s))\|_{\cL(H^{\ell+\varepsilon,\rho+\varepsilon}, H^{\ell+\varepsilon+p m^\prime,\rho+\varepsilon+\mu^\prime})}
	\cdot 
	\|f(t-\tau)\|_{H^{\ell+\varepsilon,\rho+\varepsilon}} \,d\tau
	\\
	& \le 
	B K (\max_{\tau\in[0,t]}\|f(\tau)\|_{H^{\ell+\varepsilon,\rho+\varepsilon}})\int_0^t \ta^{-1+(1-p)r}\,d\tau <\infty.
\end{align*}
We conclude $u\in C((0,+\infty), H^{\ell+m^\prime,\rho+\mu^\prime})$, as claimed.
\end{proof}
\begin{lemma}\label{lem:K1t}
The family of symbols $K_1(t)$, $t\in[0,+\infty)$, satisfies 
\[
	K_1\in C([0,\infty),S^{0,0}(\R^d\times\R^d))\cap C((0,\infty), S^{-pm^\prime,-p\mu^\prime}(\R^d\times\R^d)), \quad p\in(0,2].
\]
Moreover, for any $t\in(0,+\infty)$,
\[
	K_1\colon t\mapsto t^{-1+(1-p)r}\cdot S^{-pm^\prime,-p\mu^\prime}(\R^d\times\R^d), \quad p\in(0,2].
\]	
\end{lemma}
\begin{proof}
From Lemma \ref{lem:Falpha,beta} we know an integral representation of $E_{r,r}$. In particular, 
\[ E_{r,r}(-t^r a(x,\xi))=F_{r,r}(-t^r a(x,\xi)),\]
where
\begin{equation}
F_{r,r}(z):= \frac{ \sin(\pi r)}{\pi r }\int_0^\infty e^{-\tau^\frac{1}{r}}\frac{ \tau^{\frac{1}{r}}}{\tau^2-2\tau z\cos(\pi r)+z^2}\,d\tau.
\end{equation}
This allows to conclude
\begin{equation}\label{eq:Kasexp_nhom}
	K_1(t,x,\xi) \sim \sum_{j\geq 0} \frac{t^{jr+r-1}}{j!} A_j(x,\xi) F^{(j)}_{r,r}(-t^r a(x,\xi)),
\end{equation} 
where $F_{r,r}$ is defined by \eqref{eq:Fr1}, 
and, for  any $j\geq 1$, the function $F^{(j)}_{r,r}(-t^r a(x,\xi))$ is the derivative of order $j$ of $F_{r,r}$ evaluated in $-t^r a(x,\xi)$.

Arguing as in the proof of Lemma \ref{lem:Kt}, for any $\alpha, \beta\in \Z_+^d$ with $|\alpha|+|\beta|\geq 1$ we can apply Fa\`a di Bruno' formula to obtain
\begin{equation*}
\begin{aligned}
& \partial_x^\alpha\partial_\xi^\beta \left(F_{r,r}^{(j)}(-t^ra(x,\xi))\right) \\
&  =
\sum_{\ell=1}^{|\alpha|+|\beta|}\!\!\!\!  t^{\ell r}F^{(j+\ell)}_{r,r}\left(-t^r a(x,\xi)\right) \!\!\!\sum_{\substack{\alpha_1+\dots +\alpha_\ell=\alpha\\ \beta_1+\dots +\beta_\ell=\beta}} \!\!\!\!\tilde C^{\beta_1,\dots,\beta_\ell}_{\alpha_1, \dots, \alpha_\ell} \; \partial_{x}^{\alpha_1} \partial_\xi^{\beta_1}a(x,\xi)\cdots \partial_{x}^{\alpha_\ell} \partial_\xi^{\beta_\ell}a(x,\xi)
\end{aligned}
\end{equation*}
for suitable constants $\tilde C^{\beta_1,\dots,\beta_\ell}_{\alpha_1, \dots, \alpha_\ell}\in \R$. Applying Lemma \ref{lem:Falpha,beta} and assumptions \ref{hyp:a_hypoellyptic}, \ref{hyp:a_derivatives}, we may estimate, on $\supp(1-\chi(x,\xi))$,
\begin{equation}
\begin{aligned}
& \left|\partial_x^\alpha\partial_\xi^\beta \left(F_{r,r}^{(j)}(-t^ra(x,\xi))\right)\right| \\ &\lesssim \sum_{\ell=1}^{|\alpha|+|\beta|}  \frac{ 1}{(1+t^ra(x,\xi))^2 t^{jr}a(x,\xi)^j}   \!\!\!\!\sum_{\substack{\alpha_1+\dots +\alpha_\ell=\alpha\\ \beta_1+\dots +\beta_\ell=\beta}}\!\!\!\! \frac{|\partial_{x}^{\alpha_1} \partial_\xi^{\beta_1}a(x,\xi)|}{a(x,\xi)}\dots \frac{|\partial_{x}^{\alpha_\ell} \partial_\xi^{\beta_\ell}a(x,\xi)|}{a(x,\xi)}  \\
& 	\lesssim  \sum_{\ell=1}^{|\alpha|+|\beta|}\frac{t^{-jr}a(x,\xi)^{-j}}{(1+t^r a(x,\xi))^2}  \sum_{k=\ell}^{|\alpha|+|\beta|}\sum_{\substack{\alpha_1+\dots +\alpha_k=\alpha\\ \beta_1+\dots +\beta_k=\beta}}\<x\>^{-|\alpha_1|} \<\xi\>^{-|\beta_1|}\dots \<x\>^{-|\alpha_\ell|} \<\xi\>^{-|\beta_\ell|}\\
& \lesssim \frac{t^{-jr}a(x,\xi)^{-j}}{(1+t^r a(x,\xi))^2}\<x\>^{-|\alpha|}\<\xi\>^{-|\beta|}.
\end{aligned}
\end{equation}
As in the proof of Lemma \ref{lem:Kt}, this allows to estimate, on $\supp(1-\chi(x,\xi))$,
\begin{equation*}
|t^{jr}\partial_x^\alpha \partial_\xi^\beta \big(A_j(x,\xi)F_{r,1}^{(j)}(-t^r a(x,\xi))\big)| \lesssim \frac{\<x\>^{-j+1-|\alpha|}\<\xi\>^{-j+1-|\beta|}}{(1+t^r a(x,\xi))^2},
\end{equation*} 
for any $t\geq 0$ and $j\geq 2$, where $A_j\in S^{jm-j+1,j\mu-j+1}$ admit the asymptotic expansion in  \eqref{eq:asymptotic_Aj}. 
It follows that, for any $\alpha, \beta \in \N^d$ and $t\geq 0$, we find
\begin{equation*}
|\partial_x^\alpha \partial_\xi^\beta K_1(t,x,\xi)|\lesssim \<x\>^{-|\alpha|}\<\xi\>^{-|\beta|}.
\end{equation*}
In view of assumption \ref{hyp:a_hypoellyptic}, 
\[
	(1+t^r a(x,\xi))^2\ge (1+t^r a(x,\xi))^p, \quad p\in(0,2],t\in[0,+\infty),x,\xi\in\R^d.
\]
It follows that, for any $t>0$, arbitrarily small, again in view of assumption \ref{hyp:a_hypoellyptic}, we may likewise estimate 
\[
	|\partial_x^\alpha \partial_\xi^\beta K_1(t,x,\xi)|\lesssim t^{-1+(1-p)r} \<x\>^{-pm^\prime-|\alpha|}\<\xi\>^{-p\mu^\prime-|\beta|},\quad p\in[0,2].
\] 
As in the proof of Theorem \ref{thm:main}, the above steps also allow to derive 
the continuous dependence of $K_1$ with respect to $t\in[0,+\infty)$ or $t\in(0,+\infty)$, respectively. 
This completes the proof.
\end{proof}

We conclude with a result about the singularities of the solution for the homogeneous problem \eqref{eq:CPmain} with respect to the singularities of the initial data,
in terms of the global wavefront sets introduced in \cite{CJT13b} (see also \cite{CJT16}). With the notation introduced therein, in view of the fact that for an 
SG-ordered pair of spaces $(\cB,\cC)$, a SG-operator $A\colon\cB\to\cC$, and $u\in\cB$, 
\[ 
	\WFF_\cC(Au)\subseteq\WFF_\cB(u)\subseteq \WFF_\cC(Au)\cup\Char A,
\]
the proof of the next Theorem \ref{thm:propsing} follows immediately by Theorem \ref{thm:main}, through the representation formula \eqref{eq:reprform_u}.
\begin{theorem}\label{thm:propsing}
	Let $(\cB,\cC)$ be a SG-ordered pair of spaces with respect to the weight $\omega_{-m^\prime,-\mu^\prime}(x,\xi)=\jap{x}^{-m^\prime}\jap{\xi}^{-\mu^\prime}$.
	Under the hypotheses of Theorem \ref{thm:main}, here assuming instead $u_0\in\cB$, it follows $u\in C((0,\infty),\cC)$ and
	\[ 
		\WFF_\cC(u(t))\subseteq\WFF_\cB(u_0), \quad t\in(0,+\infty).
	\]
\end{theorem}


\section{Mittag-Leffler functions}\label{subs:M-L}
\setcounter{equation}{0}
Given $\alpha>0$ and $\beta\in \C$, we denote by $E_{\alpha,\beta}(z)$ the Mittag-Leffler function with parameters $\alpha$ and $\beta $ defined by
\[ 
	E_{\alpha,\beta}(z)=\sum_{k=0}^\infty \frac{z^k}{\Gamma(\alpha k+\beta)}.
\]
The function $E_{\alpha,\beta}$ plays a fundamental role in the theory of fractional calculus. Concerning the Laplace transform of $E_{\alpha,\beta}$ and its derivatives, the following formula can be useful (cf. equation (1.10.10) in \cite{Kilbas}): for any $\alpha>0$, $\beta,\mu \in \C$, and $j\in \N$, it holds
\begin{equation}
\label{eq:laplace_tranform_MIttagLeffler}
\mathcal{L}\bigg(\frac{t^{j\alpha+\beta-1}}{j!}E^{(j)}_{\alpha, \beta}(\mu t^\alpha)\bigg)(s)=\frac{s^{\alpha-\beta}}{(s^\alpha-\mu)^{j+1}}, \quad \text{for any } s>(\Re(\mu))^\frac{1}{\alpha},
\end{equation}
where $\displaystyle E^{(j)}_{\alpha, \beta}(z)=\left(\frac{d}{dz}\right)^{j} E_{\alpha, \beta}(z)$.
In the subsequent Lemma \ref{lem:integral_representation_MittagLeffler} we recall some useful properties of $E_{\alpha, \beta}(z)$ from \cite{Popov2013}.
\begin{lemma}
\label{lem:integral_representation_MittagLeffler}
Let $\alpha \in (0,1)$ and $\beta<1+\alpha$, and consider 
\begin{equation}\label{eq:integral_representation_MittagLeffler}
F_{\alpha, \beta}(z):= \frac{1}{\pi \alpha }\int_0^\infty \tau^{\frac{(1-\beta)}{\alpha}}e^{-\tau^\frac{1}{\alpha}}\frac{\tau \sin(\pi \beta)-z\sin(\pi(\beta-\alpha))}{\tau^2-2\tau z\cos(\pi \alpha)+z^2}\,d\tau,
\end{equation}
where the integral 
is understood in the principal value sense if $\arg(z)=\pm \pi\alpha$. Then,
the following representations hold:
\begin{align}
E_{\alpha, \beta}(z)&= F_{\alpha, \beta}(z), && \text{if} \quad  \alpha \pi <|\arg(z)|\leq \pi; \\
E_{\alpha, \beta}(z)&= F_{\alpha, \beta}(z) + \frac{1}{2\alpha}z^\frac{1-\beta}{\alpha}, && \text{if} \quad \arg(z)=\pm \pi \alpha; \\
E_{\alpha, \beta}(z)&= F_{\alpha, \beta}(z) + \frac{1}{\alpha}z^\frac{1-\beta}{\alpha}, && \text{if} \quad   |\arg(z)|< \pi \alpha.
\end{align}
\end{lemma}
\begin{lemma}
\label{lem:Falpha,beta}
For any $\alpha>0$ and $\beta<1+\alpha$ the function $F_{\alpha,\beta}=F_{\alpha,\beta}(z)$ defined by  \eqref{eq:integral_representation_MittagLeffler} satisfies 
\begin{equation}
\label{eq:Falpha,beta_estimate}
|F^{(k)}_{\alpha,\beta}(z)|\lesssim |z|^{-\frac{(\beta-1)_+}{\alpha}-k}\begin{cases}
(1+|z|)^{-1+\frac{(\beta-1)_+}{\alpha}}  \quad \text{if } \alpha\neq \beta,\\
(1+|z|)^{-2+\frac{(\beta-1)_+}{\alpha}}  \quad \text{if } \alpha= \beta,
\end{cases}
\end{equation}
for any $k\in \N_0$, where $F^{(k)}_{\alpha,\beta}(z)=d^k F_{\alpha,\beta}/d^kz$, uniformly with respect to $z\in \R_-$.
\end{lemma}
\begin{proof}
Setting $\tau^\frac{1}{\alpha}=\omega (-z)^\frac{1}{\alpha}$ (and then $\tau=-\omega^\alpha z $)we find
\begin{align*}
F_{\alpha, \beta}(z)&= \frac{1}{\pi \alpha }\int_0^\infty \tau^{\frac{(1-\beta)}{\alpha}}e^{-\tau^\frac{1}{\alpha}}\frac{\tau \sin(\pi \beta)-z\sin(\pi(\beta-\alpha))}{\tau^2-2\tau z\cos(\pi \alpha)+z^2}\,d\tau\\
&= (-z)^{\frac{1-\beta}{\alpha}}\frac{1}{\pi}\int_0^\infty e^{-\omega (-z)^\frac{1}{\alpha}}\frac{\omega^{2\alpha-\beta}\sin(\pi \beta)-\omega^{\alpha-\beta}\sin(\pi(\beta-\alpha)))}{\omega^{2\alpha}+2\omega^\alpha \cos(\pi \alpha)+1}\,d\omega.
\end{align*}
Namely,
\begin{equation}
\label{eq:Falpha,beta_representation}
\begin{aligned}
F_{\alpha,\beta}(z)&=\frac{ (-z)^{\frac{1-\beta}{\alpha}}}{\pi} \mathcal{L}(\psi_{\alpha,\beta})\left[(-z)^\frac{1}{\alpha}\right], \\ \psi_{\alpha,\beta}(\omega)&= \frac{\omega^{2\alpha-\beta}\sin(\pi \beta)-\omega^{\alpha-\beta}\sin(\pi(\beta-\alpha)))}{\omega^{2\alpha}+2\omega^\alpha \cos(\pi \alpha)+1}.
\end{aligned}
\end{equation}

We notice that $\psi:[0,+\infty)\to \R$ is continuous in $[0,+\infty)$ and  ($\mathcal{L}$-)transfor-mable with abscissa of convergence $\lambda_a(\psi)=0$ 
(see, e.g., \cite[Definition 4.1]{CGP2}).
Indeed, it holds
\begin{equation}
\label{eq:approx_psi_infty}
\text{as } \omega \to +\infty, \quad \psi_{\alpha,\beta}(\omega)\sim \begin{cases}
 \omega^{-\beta} \quad &\text{if } \beta\neq 1 \\
\omega^{-\alpha-1} \quad &\text{if } \beta= 1 
\end{cases}, 
\end{equation}
and 
\begin{equation}
\label{eq:approx_psi_0}
\text{as } \omega \to 0^+, \quad \psi_{\alpha,\beta}(\omega)\sim \begin{cases}
\omega^{\alpha-\beta} \quad &\text{ if $\alpha\neq \beta$}\\
\omega^{2\alpha-\beta} \quad &\text{ if $\alpha=\beta$}
\end{cases}
\end{equation}
and so, for any $\varepsilon>0$ there exist $c, C$ positive constants such that for any $\lambda>0$ we may estimate
\begin{align*} 
\int_0^\infty & e^{-\lambda \omega}  | \psi_{\alpha,\beta}(\omega)|d\omega \\ 
&\leq (1+\varepsilon)\left(\int_0^c \omega^{\alpha-\beta}d\omega+\int_c^C e^{-\lambda \omega} | \psi_{\alpha,\beta}(\omega)|d\omega + C^{-\beta}\int_C^\infty e^{-\lambda \omega}d\omega\right);
\end{align*}
in particular, being $\beta<1+\alpha$ it follows immediately that $e^{-\lambda \omega} \psi_{\alpha,\beta}(\omega)\in L^1([0,+\infty)$.
%
We also know that $\mathcal{L}( \psi_{\alpha,\beta})$ is holomorphic in $\C_+:=\{z\in \C: \Re(z)> 0\}$ (see, e.g., \cite[Theorem 4.4]{CGP2}); 
additionally, for any $k\in \N$ the function $\omega^k  \psi_{\alpha,\beta}(\omega)$ is 
($\mathcal{L}$-)transformable with abscissa of convergence equal to 0; in particular, it holds  
\begin{equation}
\label{eq:Laplace_derivatives_Lemma}
\frac{d^k}{ds^k}\cL (\psi_{\alpha,\beta})(s)=(-1)^k\cL (\omega^k \psi_{\alpha,\beta} )(s)
\end{equation}
for all $s\in \C$ with $\Re s>0$. Then, for $s$ which tends to $0$ two dimensionally 
in the angular region $|\arg(s)|\leq \tilde\theta<\pi/2$, the application of 
\cite[Lemma 4.19]{CGP2} and \cite[Lemma 4.20]{CGP2},
together with the estimate \eqref{eq:approx_psi_infty}, allows to conclude, for any $k\in \N$,
\begin{equation}
\label{eq:Lpsi_estimate_small}
\left| \frac{d^k}{ds^k}\cL (\psi_{\alpha,\beta})(s)\right|\leq C_k\begin{cases}
 s^{-(k-\beta+1)_+} \quad &\text{if } \beta\neq 1 \\
s^{-(k-\alpha)_+} \quad &\text{if } \beta= 1 
\end{cases}, 
\end{equation}
for some constant $C_k>0$ depending only on $\alpha>0$ and $\beta<1+\alpha$. \\
Whereas, if $s$ tends to $\infty$ two dimensionally in the angular region $|\arg(s)|\leq \theta<\pi/2$, then the application of 
\cite[Lemma 4.21]{CGP2}, 
together with the estimate \eqref{eq:approx_psi_infty}, allows to obtain 
\begin{equation}
\label{eq:Lpsi_estimate_large}
\left| \frac{d^k}{ds^k}\cL (\psi_{\alpha,\beta})(s)\right|\leq C'_k
\begin{cases}
s^{-k-\alpha+\beta-1} \quad & \text{ if $\alpha\neq \beta$}\\
s^{-k-2\alpha+\beta-1} \quad &\text{ if $\alpha=\beta$}
\end{cases} 
\end{equation}
for some constant $C'_k>0$ depending only on $\alpha>0$ and $\beta<1+\alpha$. 
Estimates \eqref{eq:Lpsi_estimate_small} and \eqref{eq:Lpsi_estimate_large} allow to derive the desired estimate \eqref{eq:Falpha,beta_estimate} for $|F^{(k)}_{\alpha,\beta}(z)|$ for any $k\in \N_0$. 
Indeed, if $k=0$ from representation \eqref{eq:Falpha,beta_representation} and estimate \eqref{eq:Lpsi_estimate_small} one obtains:
\begin{equation*}
|F_{\alpha,\beta}(z)|\lesssim 
\begin{cases} 
1 \quad & \text{ if } \beta\leq 1, \\
|z|^{\frac{1-\beta}{\alpha}} \quad & \text{ if } \beta> 1,
\end{cases}
\end{equation*}
as $|z|\to 0$ in $\R_-$; whereas, by \eqref{eq:Lpsi_estimate_large} we may estimate 
\begin{equation*}
|F_{\alpha,\beta}(z)|\lesssim \begin{cases}
|z|^{-1} \quad \text{if } \alpha\neq \beta,\\
|z|^{-2} \quad \text{if } \alpha= \beta,
\end{cases}
\end{equation*}
as $|z|\to +\infty$ in $\R_-$. Finally, we get 
\begin{equation*}
|F_{\alpha,\beta}(z)|\lesssim |z|^{-\frac{(\beta-1)_+}{\alpha}}\begin{cases}
(1+|z|)^{-1+\frac{(\beta-1)_+}{\alpha}}  \quad \text{if } \alpha\neq \beta,\\
(1+|z|)^{-2+\frac{(\beta-1)_+}{\alpha}}  \quad \text{if } \alpha= \beta,
\end{cases}
\end{equation*}
for any $z\in \R_-$.
Let us suppose $k\geq 1$; for any $j\in \{1,\dots, k\}$ it holds
\begin{equation*}
\frac{d^j}{dz^j}\left(\cL (\psi_{\alpha,\beta})\left[(-z)^\frac{1}{\alpha}\right]\right)=\sum_{i=1}^jC_{i,\alpha} \left(\frac{d^i}{ds^i}\cL (\psi_{\alpha,\beta})\right)\left[(-z)^\frac{1}{\alpha}\right](-z)^{\frac{i}{\alpha}-j},
\end{equation*}
for suitable $C_{i,\alpha}\in \R$ independent of $z$, and 
\begin{equation*}
\left|\frac{d^{k-j}}{d^{k-j}}(-z)^{\frac{1-\beta}{\alpha}}\right| \leq \tilde{C}_{k,\alpha} |z|^{\frac{1-\beta}{\alpha}-k+j}
\end{equation*}
for some $\tilde{C}_{k,\alpha}>0$ independent of $z$. Then, the application of Leibniz formula together with estimates \eqref{eq:Lpsi_estimate_small} and \eqref{eq:Lpsi_estimate_large} allows to get 
\begin{equation*}
|F^{(k)}_{\alpha,\beta}(z)|\lesssim\begin{cases} 
|z|^{-k} \quad & \text{ if } \beta\leq 1, \\
|z|^{\frac{1-\beta}{\alpha}-k} \quad & \text{ if } \beta> 1,
\end{cases}
\end{equation*}
as $|z|\to 0$ in $\R_-$, and 
\begin{equation*}
|F^{(k)}_{\alpha,\beta}(z)|\lesssim 
\begin{cases}
|z|^{-1-k}\quad \text{if } \alpha\neq \beta,\\
|z|^{-2-k}  \quad \text{if } \alpha= \beta,
\end{cases}
\end{equation*}
as $|z|\to +\infty$ in $\R_-$. Finally, for any $k\in \N_0$ and $z\in \R_-$ we can conclude the desired estimate \eqref{eq:Falpha,beta_estimate}.
\end{proof}




\bigskip  


\end{document}